\documentclass[reqno,11pt]{amsart}

\usepackage{amsmath, amssymb}
\usepackage[utf8]{inputenc}
\usepackage[T1]{fontenc}

\usepackage{xcolor}
\usepackage{graphicx}
\usepackage[hidelinks]{hyperref}
\usepackage{amscd}

\usepackage{amsthm}

\usepackage{latexsym,array}
\usepackage{amsfonts}
\usepackage{amsbsy}
\usepackage{amssymb}
\usepackage{dsfont}
\usepackage{mathrsfs}
\usepackage{mathtools}
\usepackage{enumitem}
\usepackage{float}
\usepackage{tikz}
\usepackage{tikz-cd}

\usepackage{caption}
\usepackage{subcaption}

\usepackage{color}
\usepackage{graphicx}
\usepackage[hidelinks]{hyperref}
\usepackage{cleveref}
\usepackage{fullpage}
\usepackage{comment}

\usepackage{dsfont}

\usepackage{cleveref}

\numberwithin{equation}{section}

\newcommand{\dom}{\text{dom}}

\newcommand{\BR}{\mathbb{R}}
\newcommand{\E}{\mathcal{E}}

\newtheorem{theorem}{Theorem}[section]
\newtheorem{lemma}[theorem]{Lemma}

\newtheorem{proposition}[theorem]{Proposition}

\theoremstyle{definition}
\newtheorem{remark}[theorem]{{\bf Remark}}
\newtheorem{definition}[theorem]{Definition}
\newtheorem{problem}[theorem]{Problem}

\newcommand{\midbar}{ \ \big\vert \ }
\newcommand{\p}[1]{\left( #1 \right)}

\newcommand{\set}[1]{\left\{ #1 \right\}}
\newcommand{\modulo}[1]{\left| #1 \right|}
\newcommand{\norm}[1]{\left\| #1 \right\|}
\newcommand{\dif}{\hspace*{0.5mm}d\hspace*{0.25mm}}

\newcommand{\Spin}{\textup{Spin}}
\newcommand{\spin}{\mathfrak{spin}}

\newcommand{\Sc}[1]{\textup{Sc}\p{#1}}

\newcommand{\dvol}{\: d\textup{vol}}
\newcommand{\inner}[1]{\left\langle  #1 \right\rangle }

\renewcommand{\Im}{\mathrm{Im}}

\usepackage{amscd}
\usepackage{tikz-cd}
\usepackage{tikz}

\crefname{enumi}{}{}
\crefname{enumii}{}{}

\title[]{The $S$-resolvent estimates for the Spinor Dirac operator on manifolds with boundary conditions}

\author[I. Beschastnyi]{Ivan Beschastnyi}
\address{(IB)
	Charge de Recherche, Centre Inria d'Université Côte d'Azur
} \email{ ivan.beschastnyi@inria.fr}

\author[F. Colombo]{Fabrizio Colombo}
\address{(FC)
	Politecnico di Milano\\Dipartimento di Matematica\\Via E. Bonardi, 9\\20133
	Milano, Italy}
\email{fabrizio.colombo@polimi.it}

\author[Simão Andrade Lucas]{Simão Andrade Lucas}
\address{(SL)
	Politecnico di Milano\\Dipartimento di Matematica\\Via E. Bonardi, 9\\20133
	Milano, Italy
} \email{simao.lucas@polimi.it}

\author[I. Sabadini]{Irene Sabadini}
\address{(IS)
	Politecnico di Milano\\Dipartimento di Matematica\\Via E. Bonardi, 9\\20133
	Milano, Italy
} \email{irene.sabadini@polimi.it}

\date{}

\begin{document}
\maketitle

\begin{abstract}
The aim of this paper is to show that the spectral theory based on the S-spectrum is particularly well suited for the Dirac operator on manifolds, even in cases where the operator is not self adjoint.
Traditionally, for non-self adjoint operators in the Clifford setting, the literature has often referred to the right spectrum. However, a more comprehensive approach is provided by the theory of the $S$-spectrum, which is the appropriate notion for general operators on Clifford modules. In this work, we show that this theory is particularly well suited for bisectorial Clifford operators.
By using the $S$-spectrum, which naturally contains the right eigenvalues, we prove bisectorial estimates for the $S$-resolvent associated with the spinor Dirac operator under various boundary conditions.
\end{abstract}

\vskip 1cm
\par\noindent
 AMS Classification: 47A10, 58J50.
\par\noindent
\noindent {\em Key words}: $S$-spectrum, spinor Dirac operator on manifold, Spectral properties depending on scalar curvature,
Schr\"odinger-Bochner-Weitzenb\"ock theorem.

\medskip
\textbf{Acknowledgements:} F. Colombo and I. Sabadini are supported by MUR grant Dipartimento di Eccellenza 2023-2027. I. Beschastnyi was supported by the French government through the France 2030 investment plan managed by the National Research Agency (ANR), as part of the Initiative of Excellence Université Côte d’Azur under reference number ANR-15-IDEX-01


\section{Introduction}\label{INTROD}

In this paper we present some recent results related to the spectral theory
based on the $S$-spectrum when applied to the Dirac operator on manifolds with homogeneous boundary conditions
of Dirichlet and Robin-like type.
The spectral theory on the $S$-spectrum works in general Clifford Banach modules, though we specifically formulate our problems
 within Clifford Hilbert modules, but we do not assume that the Dirac operator is necessarily self adjoint.
 This work extends the results we recently proved for the Dirac operator in hyperbolic and spherical spaces, see \cite{beschastnyi2025s}.

\medskip
The spectral theory on the $S$-spectrum represents a significant development in operator theory,
with its discovery  described in the introduction of \cite{CGK}.
Its development is associated with the theory of slice hyperholomorphic functions.
These theories are collected in numerous works and organized in the books
\cite{ACS2016,AlpayColSab2020,AlpayColSab2024,FJBOOK,CGK,ColomboSabadiniStruppa2011}.
We remark that the spectral theorems for quaternionic and Clifford operators are also based on the notion of $S$-spectrum, see \cite{ACK,ColKim}. Moreover, as described in \cite{ADVCGKS}, the $S$-spectrum generalizes far beyond the conventional quaternionic and Clifford frameworks.

\medskip
An important application of this spectral theory is within
differential geometry and the spectral theory of the Dirac operator on manifolds.
This is well studied topic in the literature were there are several contributions from different authors
and without claim completeness we referee the reader to the books \cite{Bar91,Bbw93,Cno02,Esp98,Fri00,Gin09,Hua06,Rod16} and the references therein.
In this paper we offer the point of view of the new spectral theory on the $S$-spectrum continuing
what we started in \cite{beschastnyi2025s} regarding the spectral properties in terms
 of the $S$-spectrum of the Spherical Dirac operator $\mathcal{D}_S$  and of the Hyperbolic Dirac operator $\mathcal{D}_H$.

The Dirac operator $\mathcal{D}_S$ on the spherical space, which is computed  for example in \cite{gilbert1991clifford},
and by choosing the representation given by the left multiplication in the Clifford algebra, it terns out to be
$$
\mathcal{D}_S=(1+|x|^2)\sum_{i=1}^ne_i\partial_{x_i}-nx,
$$
while the Dirac operator on the hyperbolic space, with the same representation, becomes:
\begin{equation}
		\mathcal{D}_H = \sum_{i=1}^{n-1} e_iy\partial_{x_i} - \alpha_ne_{n} + e_{n}y\partial_y,\ \ \ {\rm where} \ \ \   \alpha_n:= \frac{n-1}{2}.
	\end{equation}
In \cite{beschastnyi2025s} we study the spectral properties of $\mathcal{D}_S$ and $\mathcal{D}_H$
homogeneous Dirichlet or Robin-like boundary conditions.

\medskip
In this paper we investigate some spectral properties of the Dirac operator
on a general manifold with boundary and we associate homogeneous Dirichlet or
Robin-like boundary conditions.
Specifically, we revisit the essential concepts related to the description of Dirac operators on manifolds
and the fact that one can  write the Dirac operator in terms of the covariant derivative $\nabla_{ X}^\tau$.  Let $U$ be a coordinate neighborhood of a Riemannian manifold $(M,g)$ of dimension $n$. In this neighborhood we can find an orthonormal frame of vector fields $E_1,\dots,E_n$. Then the Dirac operator on $(U,g)$ can be written as a differential operator on
the space of sections as
\begin{equation}
\mathcal{D} = \sum_{i=1}^ne_i\nabla^\tau_{E_i}.
\end{equation}

In order to apply the spectral theory on the $S$-spectrum we need
to recall that $\mathcal{D}^2$ is given as the sum of a second-order Laplacian and a curvature operator, in accordance with the
Schr\"odinger-Bochner-Weitzenb\"ock theorem.

 \medskip
The spectral theory on the $S$-spectrum has fundamental differences compared to classical complex spectral theory, because it is derived from the classical resolvent series expansion, without assuming that the operator and the spectral parameter commute. Specifically, for bounded operators $\mathcal{T}:V\rightarrow V$ acting in a Clifford module $V$, the series expansion of the $S$-resolvent becomes
\begin{equation}\label{Eq_Resolvent_sum}
\sum_{n=0}^\infty \mathcal{T}^ns^{-n-1}=(\mathcal{T}^2-2s_0\mathcal{T}+|s|^2\mathcal{I})^{-1}
(\overline{s}\mathcal{I}-\mathcal{T}),\qquad|s|>\Vert \mathcal{T}\Vert,
\end{equation}
where $\mathcal{I}$ denotes the identity operator, $s=s_0+s_1e_1+\dots+s_ne_n$ is a paravector, $|s|$ is its modulus and $\overline{s}$ the conjugate of $s$. The reason why the value of the series \eqref{Eq_Resolvent_sum} does not simplify to the classical resolvent operator $(s\mathcal{I}-\mathcal{T})^{-1}$ is due to the noncommutativity $s\mathcal{T}\neq \mathcal{T}s$. The explicit value of the sum in equation \eqref{Eq_Resolvent_sum} motivates that, even for unbounded, closed operators $\mathcal{T}$, the spectrum must be associated with the invertibility of the operator
\begin{equation*}
Q_s[\mathcal{T}]:=\mathcal{T}^2-2s_0\mathcal{T}+|s|^2\mathcal{I},\qquad\text{with }\dom(Q_s[\mathcal{T}])=\dom(\mathcal{T}^2).
\end{equation*}
This leads us to the definitions of the \textit{$S$-resolvent set} and the \textit{$S$-spectrum}
\begin{equation*}
\rho_S(\mathcal{T}):=\big\{s\in\mathbb{R}^{n+1}\;\big|\;Q_s[\mathcal{T}]^{-1}\in
\mathcal{B}(V)\big\}\qquad\text{and}\qquad\sigma_S(\mathcal{T}):=\mathbb{R}^{n+1}\setminus\rho_S(\mathcal{T}),
\end{equation*}
where $\mathcal{B}(V)$ is the space of all bounded operators from $V$ into itself (see \cite{CGK,ColomboSabadiniStruppa2011} for more details). Motivated by \eqref{Eq_Resolvent_sum}, we define for every $s\in\rho_S(\mathcal{T})$ the \textit{left} and the \textit{right $S$-resolvent operators}
\begin{equation}\label{Eq_SL_SR}
S_L^{-1}(s,\mathcal{T}):=Q_s[\mathcal{T}]^{-1}\overline{s}-\mathcal{T}Q_s[\mathcal{T}]^{-1}
\qquad\text{and}\qquad S_R^{-1}(s,\mathcal{T}):=(\overline{s}\mathcal{I}-\mathcal{T})Q_s[\mathcal{T}]^{-1},
\end{equation}
which are then be used in the definition of the $S$-functional calculus
\begin{equation}\label{Eq_S_functional_calculus}
f(\mathcal{T}):=\frac{1}{2\pi}\int_{\partial (U\cap\mathbb{C}_J)}S_L^{-1}(s,\mathcal{T})ds_Jf(s)\quad\text{and}\quad f(\mathcal{T}):=\frac{1}{2\pi}\int_{\partial (U\cap\mathbb{C}_J)}f(s)ds_JS_R^{-1}(s,\mathcal{T}),
\end{equation}
for left, (resp. right) slice holomorphic functions $f$. This is the Clifford analogue of the Riesz-Dunford functional calculus.
 Observe that for unbounded operators $\mathcal{T}: {\rm dom}(\mathcal{T})\subset V\to V$ we have that
 $ Q_s(\mathcal{T}):  {\rm dom}(Q_s(\mathcal{T}))\subset {\rm dom}(\mathcal{T})\to V$ so that the $S$-resolvent operators are defined on $V$.

\medskip
 Sectorial, bi-sectorial or strip-operators in the Clifford setting can be defined in analogy to complex operators,
 but in this case the estimates are associated with the $S$-resolvent operators:
\begin{equation*}
\Vert S_L^{-1}(s,\mathcal{T})\Vert\leq\frac{C}{|s|}\qquad\text{and}\qquad\Vert S_R^{-1}(s,\mathcal{T})\Vert\leq\frac{C}{|s|},
\end{equation*}
where the parameter $s$ belongs to suitable sectorial, bi-sectorial or strip-type subsets of $\rho_S(\mathcal{T})$.

\medskip
{\em The content of the paper and description of the main results.}
This paper is intended for a dual audience of researchers, those working in spectral theory on the $S$-spectrum and those focused
on differential geometry and the spectral theory of the Dirac operator.
In Section \ref{Clifford algebras} we give the preliminaries on Clifford algebras
while Section \ref{DIRAC-Manifold} provides the necessary background on the Dirac operator on manifolds.
This foundational material is then used to formulate the spectral problems addressed in the subsequent sections.

In Section \ref{RESOLV-hom-dir}, we establish the $S$-resolvent estimates for the spinor Dirac operator subject to homogeneous Dirichlet boundary conditions. We investigate the spectral properties of the Dirac operator $D$ on a compact spin manifold $(M, g)$ with boundary $\partial M$, where the operator acting on the space of smooth sections $C^\infty(M,\mathcal{E})$ is defined locally by \eqref{Eq:General_Dirac} where $\mathcal{E}$  is the $\mathbb{R}_n$-module where the functions of interest takes values.
Using the Schr\"odinger-Bochner-Weitzenb\"ock formula
$$
D^2 = -\Delta_\tau + \frac{1}{4}\kappa(x),
$$
 which relates the square of the Dirac operator to the connection Laplacian $\Delta_\tau$ and the scalar curvature $\kappa(x)$, we define the spectral problem with homogeneous Dirichlet boundary conditions as follows:
\begin{equation*}
    \begin{cases}
        (D^2 - 2s_0 D + |s|^2 \mathcal{I})F = f & \text{in } M, \\
        F|_{\partial M} = 0,
    \end{cases}
\end{equation*}
for a given $f \in L^2(M,\mathcal{E})$. By applying a Green's-like formula to the Laplacian component, we derive a weak formulation associated with the bilinear form:
\begin{equation*}
    q_s(F,G) = \inner{\nabla^\tau F, \nabla^\tau G}_2 + \frac{1}{4}\inner{\kappa F, G}_2 - 2s_0\inner{DF, G}_2 + |s|^2\inner{F, G}_2,
\end{equation*}
defined on $\dom (q_s) := H_0^1(M,\mathcal E)\times H_0^1(M,\mathcal E)$. The central result of this section establishes the well-posedness of this problem and provides the $S$-resolvent estimates.

In Section \ref{RESOLV-hom-dir}
we establish the $S$-resolvent estimates with Dirichlet boundary conditions
with non-negative scalar curvature. Precisely, in Theorem \ref{Th:Main_general}
we assume that $(M,g)$ is a compact spin manifold with boundary with scalar curvature $\kappa$,
with
$$
k_{\min}:=\frac{1}{4}\min_{x\in M}\kappa(x).
$$  Under the hypothesis
\begin{equation*}
   |s|^{2}+k_{\min}-|s_{0}|^{2}C_{n}^{2} \;> 0,
\end{equation*}
 for every $f\in L^2(M,\mathcal{E})$ we show that there exists a unique weak solution $F_f\in H_0^{1}(M,\mathcal{E})$, and the $S$-resolvent satisfies the estimate
\begin{equation*}
   \norm{S_{R}^{-1}(s,D)f}_{2}\leq\frac{1}{\sqrt{|s|^{2}+k_{\min}}-|s_{0}|\sqrt{n}}\Bigl(\sqrt{n}+\frac{|s|}{\sqrt{|s|^{2}+k_{\min}}}\Bigr)\norm{f}_{2}.
\end{equation*}
In Section \ref{POINCARESECT} we consider the $S$-resolvent estimates using Poincar\'e inequality and
we distinguish the nonnegative curvature case in Theorem \ref{Th:Main_general_nonneg_POIN}, while
 the case with negative scalar curvature  is treated in Theorem \ref{Th:Main_general_neg_POIN}.

Finally, in Section \ref{ROBINSECTION}, we study the elliptic boundary value problem subject to homogeneous Robin-type boundary conditions in (\ref{EQ:BVP_ROBIN}), that is:
\begin{equation*}
    \begin{cases}
        ( D^2 - 2s_0 D + \modulo{s}^2 \mathcal{I})F = f, \quad \text{in } M \\
        (\nabla_N^\tau F + bF)\vert_{\partial M} = 0,
    \end{cases}
\end{equation*}
where $N$ is the unit normal vector field to the boundary, and $b$ is a given bounded real-valued function. Our objective is to investigate the invertibility of the operator $Q_s(D)$ in the weak sense. Specifically, we determine the conditions for which the problem \eqref{EQ:BVP_ROBIN} admits a unique weak solution for every $f \in L^2(M, \mathcal{E})$ and we show the estimates of the $S$-resolvent operator.

\medskip
We conclude by observing that Theorem~\ref{Th:Main_general}
remains valid if $(M,g)$ is taken to be a compact spin manifold without boundary (i.e., a closed manifold). In this setting, the Sobolev space $H_0^1(M,\mathcal{E})$ naturally identifies with the standard Sobolev space $H^1(M,\mathcal{E})$ since the boundary trace contributions vanish identically. Consequently, all the corresponding arguments involving the coercivity of the bilinear form $q_s$ apply directly, yielding identical $S$-resolvent estimates to those obtained in~\eqref{quatSresorig_DIR_EST_new}.

\section{Clifford algebras}\label{Clifford algebras}

To fix the notations of this paper, let us discuss the definition and main properties of Clifford algebras over the reals. They form a class of algebras that generalize the concept of complex numbers and quaternions. They are formed from a vector space equipped with a quadratic form, where the key feature is the relation between the basis vectors, typically expressed through anticommutation relations. In the following we denote by  $\mathbb{R}_n$, for $n\in\mathbb{N}$, $n\geq 2$ the
Clifford algebra generated by $n$ imaginary units $e_1,\dots,e_n$ which satisfy the relations
\begin{equation*}
e_i^2=-1\qquad\text{and}\qquad e_ie_j=-e_je_i,\qquad i\neq j\in\{1,\dots,n\}.
\end{equation*}
More precisely, $\mathbb{R}_n$ is given by
\begin{equation}
\mathbb{R}_n:=\left\{\sum\nolimits_{A\in\mathcal{A}}x_Ae_A\ \  |\ \  x_A\in\mathbb{R},\,A\in\mathcal{A}\right\},
\end{equation}
using the index set
\begin{equation*}
\mathcal{A}:=\{(i_1,\dots,i_r)\  | \ r\in\{0,\dots,n\},\ \ 1\leq i_1<\dots<i_r\leq n\},
\end{equation*}
and the basis vectors $e_A:=e_{i_1}\dots e_{i_r}$. Note, for $A=\emptyset$ the empty product of imaginary units is the real number $e_\emptyset:=1$. Moreover, we will consider the set of all paravectors
\begin{equation*}
\mathbb{R}^{n+1}:=\left\{x_0+\sum\nolimits_{i=1}^nx_ie_i \ \ |\ \  x_0,x_1,\dots,x_n\in\mathbb{R}\right\}.
\end{equation*}
For any Clifford number $x\in\mathbb{R}_n$, we define
\begin{align*}
{\rm Sc}(x)&:=x_\emptyset=x_0, && \textit{(scalar part)} \\
\overline{x}&:=\sum\nolimits_{A\in\mathcal{A}}x_A\overline{e_A}, && \textit{(conjugate)} \\
|x|:=|x|_2&:=\Big(\sum\nolimits_{A\in\mathcal{A}}|x_A|^2\Big)^{\frac{1}{2}}=({\rm Sc}(x\overline{x}))^{\frac{1}{2}}=({\rm Sc}(\overline{x}x))^{\frac{1}{2}}, && \textit{(norm)}
\end{align*}
where $\overline{e_A}=\overline{e_{i_r}}\dots\overline{e_{i_1}}$ and $\overline{e_i}=-e_i$.
The sphere of purely imaginary paravectors with modulus $1$ is defined by
\begin{equation}\label{Eq_S}
\mathbb{S}:=\big\{s\in\mathbb{R}^{n+1}\;\big|\;s_0=0,\,|s|=1\big\}.
\end{equation}
Any element $J\in\mathbb{S}$ satisfies $J^2=-1$ and hence the corresponding hyperplane
\begin{equation*}
\mathbb{C}_J:=\big\{x+Jy\;\big|\;x,y\in\mathbb{R}\big\}
\end{equation*}
is an isomorphic copy of the complex numbers. Moreover, for every paravector $s\in\mathbb{R}^{n+1}$ we consider the corresponding  $(n-1)$-sphere
\begin{equation*}
[s]:=\big\{x_0+J|\Im(s)|\;\big|\;J\in\mathbb{S}\big\}.
\end{equation*}
A subset $U\subseteq\mathbb{R}^{n+1}$ is called axially symmetric, if $[s]\subseteq U$ for every $s\in U$. \medskip

\begin{remark}
We remark that in some computations in order to avoid confusion the Euclidean norm $|x|$ of $x\in \mathbb{R}_n$ will be denoted by $|x|_2$.
\end{remark}
It is now obvious that for any Clifford number $x\in\mathbb{R}_n$ one can calculate its coefficients $x_A$ by
\begin{equation}\label{Eq_xA}
{\rm Sc}(x\overline{e_A})={\rm Sc}\Big(\sum\nolimits_{B\in\mathcal{A}}x_Be_B\overline{e_A}\Big)=\sum\nolimits_{B\in\mathcal{A}}x_B{\rm Sc}(e_B\overline{e_A})=x_A,
\end{equation}
where in the last equation we used that
\begin{equation}\label{Eq_SceBeA}
{\rm Sc}(e_B\overline{e_A})=\begin{cases} 1, & \text{if }B=A, \\ 0, & \text{if }B\neq A. \end{cases}
\end{equation}

Clifford (Hilbert) modules arise in the context of Clifford algebras. Specifically, a Clifford module is a vector space equipped with a linear action of a Clifford algebra.
This action is typically defined such that the elements of the Clifford algebra act on the module in a way that respects the algebra's multiplication rules.
Clifford modules are particularly important to provide a framework for studying the representations of the Dirac operator.
For any real Hilbert space $V_\mathbb{R}$ with inner product $\langle\cdot,\cdot\rangle_\mathbb{R}$ and norm $\Vert\cdot\Vert_\mathbb{R}^2=\langle\cdot,\cdot\rangle_\mathbb{R}$, we define the Clifford- Hilbert module
\begin{equation*}
V:=\left\{\sum\nolimits_{A\in\mathcal{A}}F_A\otimes e_A\ \  |\ \  F_A\in V_\mathbb{R}\right\}.
\end{equation*}
For any vector $F=\sum_{A\in\mathcal{A}}F_A\otimes e_A\in V$ and any Clifford number $x=\sum_{A\in\mathcal{A}}x_Ae_A\in\mathbb{R}_n$, we equip the module $V$ with a left and a right scalar multiplication
\begin{subequations}
\begin{align}
xF:=&\sum\nolimits_{A,B\in\mathcal{A}}(x_BF_A)\otimes(e_Be_A), && \textit{(left-multiplication)} \\
Fx:=&\sum\nolimits_{A,B\in\mathcal{A}}(F_Ax_B)\otimes(e_Ae_B). && \textit{(right-multiplication)} \label{Eq_Right_multiplication}
\end{align}
\end{subequations}
The element $F=\sum_{A\in\mathcal{A}}F_A\otimes e_A\in V$ is usually
denoted by
$
F=\sum_{A\in\mathcal{A}}F_A e_A
$
for simplicity.
Moreover, we define the \textit{inner product}
\begin{equation}\label{Eq_Inner_product}
\langle F,G\rangle :=\overline{F}G=\sum\nolimits_{A,B\in\mathcal{A}}\langle F_A,G_B\rangle_\mathbb{R}\,\overline{e_A}e_B,\qquad F,G\in V,
\end{equation}
The sesquilinear form \eqref{Eq_Inner_product} is clearly right-linear in the second, and right-antilinear in the first argument, i.e. for every $F,G,H\in V$, $x\in\mathbb{R}_n$, there holds
\begin{align*}
\langle F,G+H\rangle &=\langle F,G\rangle +\langle F,H\rangle , && \langle G,Hx\rangle =\langle G,H\rangle  x, \\
\langle F+H,G\rangle &=\langle F,G\rangle +\langle H,G\rangle , && \langle Gx,H\rangle =\overline{x}\langle G,H\rangle .
\end{align*}
Moreover, the following property also holds
\begin{equation}\label{Eq_Inner_product_property}
\langle G,xH\rangle =\langle\overline{x}G,H\rangle .
\end{equation}
From the inner product (\ref{Eq_Inner_product}),
taking the scalar part gives
\begin{equation*}
{\rm Sc}\langle F,G\rangle :={\rm Sc}(\overline{F}G)={\rm Sc}\Big(\sum\nolimits_{A,B\in\mathcal{A}}\langle  F_A,G_B\rangle_\mathbb{R}\,\overline{e_A}e_B\Big),\qquad F,G\in V,
\end{equation*}
and setting $G=F$ we get the norm
\begin{equation}\label{Eq_Norm}
\Vert F\Vert :=\big({\rm Sc}\langle F,F\rangle \big)^{\frac{1}{2}}=
\Big(\sum\nolimits_{A\in\mathcal{A}}\Vert F_A\Vert_\mathbb{R}^2\Big)^{\frac{1}{2}}
,\qquad F\in V.
\end{equation}

\begin{remark}
In the vector space $V$, we consider two inner products that serve different purposes:
\begin{itemize}
\item[(I)]
 $\langle \cdot,\cdot\rangle : V \times V \to \mathbb{R}_n $ is the one in (\ref{Eq_Inner_product}),
 \item[(II)]
   $\mathrm{Sc}\langle \cdot,\cdot\rangle : V \times V \to \mathbb{R}$  is the scalar part of
   $ \langle \cdot,\cdot\rangle $.
  \end{itemize}
 These two inner products serve for different purposes.
The inner product $\langle \cdot,\cdot\rangle $ is used in the
 Riesz's representation theorem:
 every linear and continuous functional $\varphi: V \to \mathbb{R}_n$  is
  represented by an element $F_\varphi \in V$ such that
  $\varphi(G) = \langle F_\varphi, G \rangle $ for every $G\in V$.
The inner product $\mathrm{Sc}\langle \cdot,\cdot\rangle : V \times V \to \mathbb{R}$ is important since, unlike  $\langle \cdot,\cdot\rangle : V \times V \to \mathbb{R}_n $, it gives rise to a norm in the classical sense.
\end{remark}

Next we recall some well known properties of the inner product \eqref{Eq_Inner_product} and the norm \eqref{Eq_Norm} in the following:

\begin{lemma}\label{lem_Properties}
For every $F,G\in V$, $x\in\mathbb{R}_n$, we have:

\begin{enumerate}
\item[i)] $\Vert Fx\Vert \leq 2^{\frac{n}{2}}|x| \Vert F\Vert $
\qquad and\qquad$\Vert xF\Vert \leq 2^{\frac{n}{2}}|x| \Vert F\Vert $,
\item[ii)] $\Vert Fx\Vert =\Vert xF\Vert =|x| \Vert F\Vert $
 for $x\in\mathbb{R}^{n+1}$,
\item[iii)] $|\langle F,G\rangle |\leq 2^{\frac{n}{2}}\Vert F\Vert \,\Vert G\Vert $,
\item[iv)] $|{\rm Sc}\langle F,G\rangle |\leq\Vert F\Vert \,\Vert G\Vert $.
\end{enumerate}
\end{lemma}

\section{Preliminaries on Dirac operator on manifolds}\label{DIRAC-Manifold}

In this section we collect the necessary concepts for the local description of Dirac operators on manifolds. To give a precise definition of the spectrum of the Dirac operator on $M$, precisely the $S$-spectrum, we first need the
definition of the Dirac operator $D$ on $M$ as well as
the  identification of $D^2$ as the sum of a second-order Laplacian and a curvature operator (Schr\"odinger-Bochner-Weitzenb\"ock theorem). The precise  expression of the operator $D^2$ is of crucial importance in order to define the $S$-spectrum because it is associated with the operator
$$
\mathcal{Q}_s(D):=D^2-2s_0D+|s|^2\mathcal{I}.
$$

Let us recall for completeness the main definitions and constructions needed to define a Dirac operator on $M$. First recall that the standard Dirac operator on the Euclidean space $\mathbb{R}^n$ is usually written in the form:
\begin{equation}
    \label{eq:Dirac_Euclidean}
    DF(x) = \sum_{i=1}^n e_i \partial_i F (x)
\end{equation}
In order to generalize this setting to manifolds we need to observe what are the main ingredients in this definition:
\begin{itemize}
    \item A Clifford algebra $\mathbb{R}_n$ with generators $e_i$, $1\leq i \leq n$;
    \item An $\mathbb{R}_n$-module $\mathcal{E}$ where the function $F$ of interest takes values;
    \item A directional derivative operator $\partial_i$.
\end{itemize}
On a Riemannian manifold $(M,g)$ we must find a counterpart for each of the ingredients. To do this we must work with sections of various vector bundles. More precisely, we will need:
\begin{itemize}
    \item A Clifford algebra bundle $Cl(TM)$ over $M$, where $TM$ is the tangent bundle. Local generators $e_i$ will be sections of this bundle and their linear span will give rise to a sub-bundle of $Cl(TM)$ isomorphic to $TM$;
    \item A vector bundle $\mathcal{E}$ for which each fiber $S_x$ is a $Cl(TM)_x$ module. Sections of this bundle will play the role of the function $F$ in~\eqref{eq:Dirac_Euclidean};
    \item A covariant derivative on $Cl(TM)$ and a covariant derivative on $\mathcal{E}$, which satisfy certain compatibility conditions. We can see that in the formula~\eqref{eq:Dirac_Euclidean} we only have derivatives of $F$. However, when we write down $D^2$ we also need a way to differentiate $e_i$'s, which will be sections of $Cl(TM)$.
\end{itemize}
When all of the ingredients are assembled together an invariant definition of the Dirac operator can be given.

         We start with the construction of relevant vector bundles using spin structures. Consider an orientable Riemannian manifold $(M,g)$.
         Let $P_{SO}(TM)$ be the bundle of oriented orthonormal frames of $TM$. Recall that for $n\geq 3$ we have the universal covering homomorphism $\xi_0:\Spin(n)\longrightarrow SO_n$ with kernel $\set{1,-1}\cong \mathbb{Z}_2$. Therefore, we use this fact to build a principal $\Spin(n)$-bundle.

        \begin{definition}
            Suppose $n\geq 3$. Then a spin structure on $TM$ is a principal $\Spin(n)$-bundle $P_{\Spin}(TM)$ together with a $2$-sheeted covering
            \begin{equation*}
                \xi: P_{\Spin}(TM)\longrightarrow P_{SO}(TM)
            \end{equation*}
            such that $\xi(pg)=\xi(p)\xi_0(g)$ for all $p\in P_{\Spin}(TM)$ and all $g\in\Spin(n)$.
        \end{definition}

        The spin structure might not be unique, or not even exist. In fact, if it exists, the second Stiefel-Whitney class of $TM$, denoted by $w_2(TM)$, vanishes, i.e., $w_2(TM) = 0$~\cite[Chapter 2, Theorem 1.7]{lawsonmichelson_spingeometry}.
        From now on, we will consider only manifolds that have a spin structure.
Let us now construct the Clifford algebra bundle and the Clifford module bundle that we use in this article. Consider the Clifford algebra $\mathbb R_n$ and the spin group $\Spin(n)$ as a subset of $\mathbb R_n$. As a result we have two natural representations. The first one is the so called left regular representation
    	\begin{equation} \label{Eq:Spin_Rep}
    		\tau: \Spin(n) \longrightarrow \text{End}(\BR_n)
    	\end{equation}
of $\Spin(n)$ on $\BR_n$ given by the left multiplication on $\BR_n$, i.e. defined as  $$\tau(a)(u):=au$$ for any $a\in\Spin(n)$ and $u\in\BR_n$. The second one, that we will call conjugation, is
$$\sigma:\Spin(n) \longrightarrow \text{End}(\BR_n)$$ given by
$$
\sigma(a)u := aua^{-1}.
$$
For these two representations we can construct the corresponding associated bundles
\begin{align*}
    Cl(TM) &= P_{Spin}(TM)\times_\sigma \BR_n,\\
    \mathcal{E} &= P_{Spin}(TM)\times_\tau \BR_n.
\end{align*}
Fibers $\mathcal{E}_x$ are natural left $Cl(TM)_x$-modules. This gives us a Clifford multiplication map
\begin{equation}\label{cimutiplic}
c: Cl(TM) \otimes \mathcal{E} \to \mathcal{E}.
\end{equation}
Indeed, recall that sections of an associated bundle are in one-to-one correspondence with equiregular maps from the principal bundle to the representation spaces. For example, the sections of $\mathcal{E}$ are in one-to-one correspondence with functions in $C^\infty(P_{Spin}(TM),\BR_n)$, which satisfy
$$
f(a\cdot x) = \tau(a^{-1}) f(x) = a^{-1}f(x),
$$
while sections of $Cl(TM)$ are represented by functions in $C^\infty(P_{Spin}(TM),\BR_n)$ which satisfy
$$
g(a\cdot x) =\sigma(a^{-1}) g(x) = a^{-1} g(x) a.
$$
Therefore, $g(x)f(x)$ is equivariant with respect to the representation $\tau$.
The next step is to construct connections on those two bundles.
	\begin{definition}\label{1.8}
		A connection on a vector bundle $\mathcal{E}\longrightarrow M$ is a bilinear map that assigns to a vector field $X\in C^\infty(M,TM)$ and a section $Y \in C^\infty(M,\mathcal{E})$  a section $\nabla_X(Y)$, and which satisfies the conditions
		\begin{itemize}
			\item[(i)]
			$\nabla_{f_1 X_1 + f_2 X_2}(Y)=f_1\nabla_{ X_1}(Y)+f_2\nabla_{X_2}(Y)$,
			\item[(ii)]
			$\nabla_X(fY)=(Xf)Y+f\nabla_{X}(Y)$,
		\end{itemize}
		for all $X_1,X_2\in C^\infty(M,TM)$, $Y,\in \mathcal{C}^\infty(M,\mathcal{E})$ and for all $f,g\in \mathcal{C}^\infty(M)$. Section $\nabla_X(Y)$ is called the covariant derivative of $Y$ along $X$.
	\end{definition}

Recall that every Riemannian manifold $(M,g)$ has a canonical connection, namely the Levi-Civita connection.

\begin{definition}
    A connection $\nabla$ on $TM$ is called torsion free if it satisfies
    $$
    \nabla_X Y - \nabla_Y X = [X,Y], \qquad \forall X,Y \in C^\infty(M,TM).
    $$
    A connection $\nabla$ on $TM$ is called metric compatible if
    $$
    X(g(Y,Z)) = g(\nabla_X(Y),Z)+g(Y,\nabla_X Z).
    $$
\end{definition}

\begin{theorem}[Theorem 1.10,~\cite{gilbert1991clifford}]
    Every Riemannian manifold admits a unique torsion free metric compatible connection $\nabla$. This connection is called the Levi-Civita connection.
\end{theorem}

Locally in a chart $U \subset M$, we can choose a basis $E_i \in C^\infty(M,TM)$, $i=1,\dots,n$ of vector fields. Then there exist functions $\nu_{ij}^k\in C^\infty(M)$ such that
\begin{equation}
    \label{eq:connection_coefficients}
    \nabla_{E_i}E_j(x) = \sum_{k=1}^n \nu_{ij}^k(x) E_k(x).
\end{equation}
These functions are called connection coefficients or Christoffel symbols associated to the connection $\nabla$ and the frame $E_1,\dots E_n$. Knowing the connection coefficients, allows us to compute the covariant derivative of any section of $C^\infty(M,TM)$ along any other section using the rules from Definition~\ref{1.8}.

Let us make the description of the Levi-Civita connection $\nabla$ more explicit. Fix a chart $U\subset M$ and choose an orthonormal basis of vector fields $E_1,\dots,E_n$ defined on $U$. Since the commutator of vector fields is again a vector field, we have
$$
[E_i,E_j](x) = (E_i \circ E_j - E_j \circ E_i)(x) =: \sum_{k=1}^n c_{ij}^k(x) E_k(x).
$$
Locally defined functions $c^k_{ij}\in C^\infty(U)$ are called structure functions and they can be used to compute the connection coefficients~\eqref{eq:connection_coefficients} of the Levi-Civita connection.

	\begin{theorem}[ ] \label{Th:ChrisSymb}
		The connection coefficients associated to an orthonormal frame $E_1,\dots,E_n$ are given by
		\begin{equation*}
			\nu_{ij}^k = \frac{1}{2}\p{c_{ij}^k - c_{jk}^i + c_{ki}^j}, \qquad i,j,k = 1,\dots,n,
		\end{equation*}
		where $c_{ij}^k$ are the structure functions of the frame $E_1,\dots,E_n$.
	\end{theorem}

The Levi-Civita connection induces a natural connection on the orthonormal frame bundle $P_{SO}(TM)$ see \cite{lawsonmichelson_spingeometry}, that can be further lifted to the spin bundle $P_{Spin}(TM)$. Since $Cl(TM)$ and $\mathcal{E}$ are associated bundles, a connection on $P_{Spin}(TM)$ further induces connections $\nabla^\tau$ and $\nabla^\sigma$ on $\mathcal{E}$ and $Cl(TM)$ correspondingly. The two connections are compatible in the following sense.
\begin{proposition}[\cite{lawsonmichelson_spingeometry}, Proposition 4.11]
Given a section $F\in C^\infty(M,\mathcal{E})$ and a section $G\in C^\infty(M,Cl(TM))$, one has
$$
\nabla^\tau (G\cdot F) = (\nabla^\sigma G)\cdot F+ G\cdot\nabla^\tau ( F).
$$
\end{proposition}

Let us write down explicitly the connection $\nabla^\tau$. Since the structure functions of an orthonormal basis $E_1,\dots,E_n$ have anti-symmetric in the lower indices, i.e., $c^k_{ij} = - c^k_{ji}$, the connection coefficients of a Levi-Civita connection given by Theorem~\ref{Th:ChrisSymb} are anti-symmetric in the upper and lower indices, that is
$$
\nu_{ij}^k(x) = - \nu_{ik}^j(x).
$$
Therefore, if we fix an index $i\in\{1,\dots,n\}$, we can view $\nu_{ij}^k$ as components of a skew-symmetric $n\times n$ matrix. Therefore, we get an $\mathfrak{so}(n)$-valued function that we call $\nu_i\in C^\infty(U,\mathfrak{so}(n))$. Since $\mathfrak{so}(n)$ is isomorphic to $\spin(n)$, we can give a spin representation of the same function:
$$
\nu_i(x) = \frac{1}{2}\sum_{1\leq j < k \leq n} \nu_{ij}^ke_je_k = \frac{1}{4}\sum_{j, k=1}^n \nu_{ij}^k e_je_k.
$$

Now we can write a local expression for the connection operator $\nabla^\tau$. For an orthonormal basis $E_1,\dots,E_n$ of vector fields we have, see \cite{lawsonmichelson_spingeometry}
$$
\nabla^\tau_{E_i} = E_i+d\tau(\nu_i(x)).
$$
A covariant derivative with respect to any other vector field can be computed from the properties in the Definition~\ref{1.8}.
Recall that representation $\tau$ gives rise to a representation
$$
\dif\tau:\spin(n)\to \mathfrak{end}(\BR_n).
$$
 The Lie algebra $\spin(n)$ can be identified with the subspace of $\BR_n$ generated by $e_ie_j$ for $1\leq i <j\leq n$, see \cite{lawsonmichelson_spingeometry}. The representation $d\tau$ is then just
$$
\dif\tau(e_ie_j)u = e_ie_j u.
$$
Which gives
$$
\nabla^\tau_{E_i} = E_i+\frac{1}{2}\sum_{1\leq j<k\leq n}\nu^k_{ij}e_je_k.
$$
 We proceed next in defining the spinor Dirac operator. The expression~\eqref{eq:Dirac_Euclidean} uses derivatives along an orthonormal basis in $\BR^n$. Since on a manifold such a choice can be only made locally, one uses instead the total derivative map
 $$
\nabla^\tau: C^\infty(M,\mathcal{E}) \to C^\infty(M,\mathcal{E}\otimes T^* M).
$$
Assume that we have chosen a local basis $E_1,\dots,E_n$ of sections of $TM$, and a local basis of the dual forms $\theta^1,\dots, \theta^n \in C^\infty(M,T^*M)$, meaning that $\theta^i(E_j)=\delta^i_j$. Then the local expressions for the total derivative map is given by
$$
\nabla^\tau F = \sum_{i=1}^n \nabla_{E_i}^\tau F \otimes \theta^i.
$$
Since $\BR^n \subset \BR_n$ is invariant under conjugation, there is a sub-bundle of $Cl(TM)$ that is isomorphic to $TM$. In addition, since $M$ is endowed with a Riemannian metric, the tangent bundle $TM$ is isomorphic to $T^* M$, where the isomorphism is the usual isomorphism
 (this is Riesz representation theorem). Therefore, we can identify $T^* M$ with a sub-bundle of $Cl(TM)$.
Using these identifications the spinor Dirac operator is defined as
$$
D = c \circ \nabla^\tau,
$$
where $c$ is defined in (\ref{cimutiplic}).
Given a basis of sections $E_1,\dots,E_n$ of $TM$, we can write  the Dirac operator locally as
$$
D = \sum_{i=1}^n e_i \nabla^\tau_{E_i},
$$
where $e_1,\dots e_n$ are sections of $Cl(TM)$ that correspond to sections $E_i$ of $TM\subset Cl(TM)$.
When $E_i$ are orthonormal, then we get an expression of the form
\begin{equation} \label{Eq:General_Dirac}
    		D=\sum_{i=1}^ne_i\nabla^\tau_{E_i}=\sum_{i=1}^ne_i \p{E_i(x)+\frac{1}{2}\sum_{1\leq j<k\leq n}^n\nu_{ij}^k(x)e_je_k},
    	\end{equation}
\begin{remark}[Equivariance and Representation Theory]
The validity of algebraic operations on associated bundles is determined by their equivariance with respect to the underlying representations. An operation is well-defined on the bundle only if it commutes with the representation maps.
Consider the multiplication map $m: \BR_n \otimes \BR_n \to \BR_n$. For the Clifford bundle $Cl(TM)$, which is associated with the representation $\sigma$, the following diagram commutes:
\[
\begin{tikzcd}
\BR_n \otimes \BR_n \arrow[r, "m"] \arrow[d, "\sigma \otimes \sigma"'] & \BR_n \arrow[d, "\sigma"] \\
\BR_n \otimes \BR_n \arrow[r, "m"]                                   & \BR_n
\end{tikzcd}
\]
This confirms that multiplication is a valid operation on $Cl(TM)$. However, for the bundle $\E$, associated with the representation $\tau$, the corresponding diagram is not commutative:
\[
\begin{tikzcd}
\BR_n \otimes \BR_n \arrow[r, "m"] \arrow[d, "\tau \otimes \tau"'] & \BR_n \arrow[d, "\tau"] \\
\BR_n \otimes \BR_n \arrow[r, "m"]                                 & \BR_n
\end{tikzcd}
\]
This failure of commutativity provides the formal justification for the lack of a well-defined internal multiplication on $\E$.
\end{remark}

Next we would like to define a scalar product on $\mathcal{E}$ that would generalize the scalar product given by $Sc(\bar F G)$.
We can construct a bilinear map
$$
C^\infty (M,\mathcal{E}) \otimes C^\infty (M,\mathcal{E}) \to C^\infty(M)
$$
as follows. Since sections of $\mathcal{E}$ are in one-to-one correspondence with smooth equivariant functions $C^{\infty}_{eq}(P_{Spin},\BR_n)$ to a pair of sections $F,G \in C^{\infty} (M,\mathcal{E})$ we can associate a pair of functions $f,g\in C^{\infty}_{eq}(P_{Spin},\BR_n)$. Then we can define a map
\begin{equation}
    \label{eq:scalar_preproduct}
    F \otimes G \mapsto  {\rm Sc}(F(x),G(x)) := {\rm Sc}(\overline{f}(p)g(p)), \qquad \pi(p) = x.
\end{equation}
Since $C^{\infty}_{eq}(P_{Spin},\BR_n) \subset C^\infty(P_{Spin},\BR_n)$, from the definition of the representation $\tau$ it follows that the function on the right is invariant with respect to the action on the spin group. This means that it descends to a function on $M$. It follows that $(F(x),G(x))$ is real-valued, symmetric and non-degenerate, since the same is true for $Sc(\bar f g)$. Thus we get a natural metric on $\mathcal{E}$. Moreover, the connection $\nabla^\tau$ is compatible with this metric. Using the Riemannian volume, this allows us to define an $L^2$-scalar product on $\mathcal{E}$ as
\begin{equation}
    \label{eq:L2_norm_sections}
    \langle F,G\rangle_2 = \int_M  {\rm Sc}(F(x),G(x))d vol
\end{equation}
where $d vol$ is the Riemannian volume form.
In local coordinates it is usually written as $\sqrt{\det g} \, \dif x_1 \wedge \cdots \wedge \dif x_n$.
This allows us to define $L^2$-sections of $\mathcal{E}$ as well as the adjoint for the total derivative operator $\nabla^\tau$.

        \begin{theorem}[Schr\"odinger-Lichnerowicz formula]
            Let $M$ be a spin manifold and suppose $\mathcal{E}$ is a real spinor bundle over $M$ with $\BR_n$ as fiber endowed with the canonical Riemannian connection. Then
            \begin{equation}
            \label{eq:laplace_connection}
                D^2 = \Delta_\tau + \frac{1}{4}\kappa,
            \end{equation}
            where  $\Delta_\tau:=\p{\nabla^\tau}^*\nabla^\tau$ is the connection Laplacian on $\mathcal{E}$ and $\kappa$ is the scalar curvature.
        \end{theorem}

We now introduce two other spaces crucial for what follows:
	\begin{align*}
		C_c^k(M,\mathcal{E}) & := \set{F\in C^k(M,\mathcal{E}) \midbar \operatorname{supp }(F) \text{ is compact}}, \\
		C_{cc}^k(M,\mathcal{E}) & := \set{F\in C_c^k(M,\mathcal{E}) \midbar \operatorname{supp }(F)\subset\mathring{M}}.
	\end{align*}
	In other words, a general $G\in C_c^k(M,\mathcal{E})$ might satisfy $\operatorname{supp }(F)\cap \partial M \ne \emptyset$. On the contrary, $F\in C_{cc}^k(M,\mathcal{E})$ always satisfies $\operatorname{supp }(F)\cap\partial M = \emptyset$.
	
\begin{theorem}[\cite{lawsonmichelson_spingeometry}, Chapter II, see proof of Prop. 8.1] \label{Th:Green_like_formula}
    The operator
    $$-\Delta_\tau: C_c^\infty(M, \mathcal{E}) \longrightarrow C_c^\infty(M, \mathcal{E})$$ satisfies the Green's-like formula:
    \begin{equation} \label{Eq:Green's}
        \int_{M} \Sc{-\Delta_\tau F, G} \: \dvol = \int_{M} \Sc{\nabla^\tau F, \nabla^\tau G} \: \dvol - \int_{\partial M} \Sc{\nabla_N^\tau F, G} \: \widetilde{\dvol},
    \end{equation}
    where $F, G \in C_c^\infty(M, \mathcal{E})$, $N$ is the unit normal vector field to the boundary, $$\dvol := \sqrt{\det g} \, \dif x_1 \wedge \cdots \wedge \dif x_n,$$ and $\widetilde{\dvol}$ is the restriction of $\dvol$ to the boundary.
\end{theorem}

Given $F\in C^\infty(M,\mathcal{E})$, we have that $\nabla^\tau F\in C^\infty(M,T^*M\otimes \mathcal{E})$. The bundle $T^* M$ is a Riemannian bundle with a metric given by the inverse of the metric $g$, as well as $\mathcal{E}$, for which we have defined the metric in~\eqref{eq:scalar_preproduct}. Therefore, $T^*M\otimes \mathcal{E}$ is also a Riemannian bundle. We need this fact to define the scalar product $\langle \nabla^\tau F,\nabla^\tau G\rangle_2$. In terms of a local orthonormal frame of vector fields $(E_1,\ldots,E_n)$, we have
\begin{equation}\label{NN2}
 {\rm Sc}(\nabla^\tau F,\nabla^\tau G)(x)
:=\sum_{i=1}^n  {\rm Sc}(\nabla_{E_i}^\tau F,\nabla_{E_i}^\tau G)(x).
\end{equation}
In addition to~\eqref{eq:L2_norm_sections} we set
$$
\inner{\nabla^\tau F,\nabla^\tau G}_2:
=  \int_{M}  {\rm Sc}(\nabla^\tau F,\nabla^\tau G)(x) \dvol
$$
and we define the scalar product
$$
\inner{F,G}_{H^1}:= \inner{ F,G}_2+ \inner{\nabla^\tau F,\nabla^\tau G}_2.
$$
The norm associated with $\inner{F,G}_{H^1}$ is defined by
	\begin{equation}\label{normHone}
		\norm{F}_{H^1}^2 := \inner{ F,F}_2+ \inner{\nabla^\tau F,\nabla^\tau F}_2= \|F\|_2^2+
\|\nabla^\tau F\|_2^2.
	\end{equation}

We define $H^1(M,\mathcal{E})$ as the completion of $C^\infty(M,\mathcal{E})$ with respect to this norm.
	Additionally, we can define the space $H_0^1(M,\mathcal{E})$ as the closure of $C_{cc}^\infty(M,\mathcal{E})$ in the norm $\norm{\cdot}_{H^1}$.

\section{The $S$-resolvent estimates with homogeneous Dirichlet boundary condition}\label{RESOLV-hom-dir}

	In this section, we obtain a similar results as in \cite{beschastnyi2025s}, but for a compact spin manifold $M$ with boundary $\partial M$ and any Riemannian metric $g$ that we denote by $(M,g)$. We study the spectral properties on the $S$-spectrum for the spinor Dirac operator $D$, acting on $C^\infty(M,\mathcal E)$.
	The spectral problem we consider refers  to the $S$-spectrum
so it involves also the square of the Dirac operator $D$:
	\begin{equation} \label{EQ:Op_Qs}
		Q_s(D) = D^2 - 2s_0 D + \modulo{s}^2.
	\end{equation}

	We will study the elliptic boundary value problem with homogeneous Dirichlet boundary condition
	\begin{equation} \label{EQ:BVP_Dirichlet}
		\begin{cases}
			Q_s(D)F = f, \quad \text{in } M \\
			F\vert_{\partial M} = 0,
		\end{cases}
	\end{equation}
	and the invertibility of the operator $Q_s(D)$ in (\ref{EQ:Op_Qs}) in the weak sense. That is, we ask wether the problem (\ref{EQ:BVP_Dirichlet}) admits a unique weak solution for every $f\in L^2(M,\mathcal E)$. Since the operator $Q_s(D)$ is elliptic, by standard properties of elliptic equations, the weak solution is also a strong solution.

	In order to derive the weak formulation, we have to integrate by parts the second order part of $Q_s(D)$. By~\eqref{eq:laplace_connection} we have
    $$
    Q_s(D)F = \Delta_\tau F + \frac{\kappa(x)}{4}F - 2s_0DF + |s|^2 F.
    $$
    Next we apply the Green's-like formula for $-\Delta_\tau$ of Theorem \ref{Th:Green_like_formula}. Therefore, for $F,G\in H_0^1(M,\mathcal E)$ and for every $s\in\BR^{n+1}$, we consider the following bilinear form associated with the boundary value problem (\ref{EQ:BVP_Dirichlet})
	\begin{align} \label{EQ:Form_general}
		q_s(F,G) & = \int_M \Sc{\overline{\nabla^\tau F} \nabla^\tau G} \: \dif\text{vol}
+ \frac{1}{4}\int_{M} \kappa(x)\Sc{\overline{F}G} \: \dif\text{vol} \nonumber \\ & \hspace*{4cm} - 2s_0\int_{M} \Sc{\overline{DF}G} \: \dif\text{vol} + \modulo{s}^2\int_{M} \Sc{\overline{F}G} \: \dif\text{vol}
	\end{align}
	obtained by applying Theorem \ref{Th:Green_like_formula} to the second order part of $Q_s(D)$ in (\ref{EQ:BVP_Dirichlet}), i.e. to $\Delta_\tau$. Note that we are considering the Dirichlet problem (\ref{EQ:BVP_Dirichlet}), so the boundary term of (\ref{Eq:Green's}) is zero.
	
	\begin{problem}\label{PROB_DIR_FULL_NORM_DIR}
		Let $D$ be the Dirac operator in $\eqref{Eq:General_Dirac}$ and $\kappa$ the scalar curvature. Consider the inner product $\inner{F,G}_2$ in $L^2$. We write the bilinear form $q_s(F,G)$ defined in (\ref{EQ:Form_general}) as
		\begin{equation} \label{EQ:Form_gen_inner}
			q_s(F,G) = \inner{\nabla^\tau F,\nabla^\tau G}_2 + \frac{1}{4}\inner{\kappa F,G}_2 - 2s_0\inner{DF,G}_2 + \modulo{s}^2\inner{F,G}_2,
		\end{equation}
		with $\dom (q_s) := H_0^1(M,\mathcal E)\times H_0^1(M,\mathcal E)$.
Show that for some values of the spectral parameter $s\in\BR^{n+1}$, for every $f\in L^2(M,\mathcal E)$ there exists a unique solution $F_f\in H_0^1(M,\mathcal E)$ such that
		\begin{equation}\label{WEAK_EQUATION_DIRIC}
			q_s(F_f,G) = \inner{f,G}_2, \quad \text{for all } G\in H_0^1(M,\mathcal E).
		\end{equation}
		Furthermore, determine $L^2$- and $\nabla^\tau$-estimates of $F_f$, depending on the parameter $s\in\BR^{n+1}$.
	\end{problem}
\begin{remark}
For Problem \ref{PROB_DIR_FULL_NORM_DIR} we endow $H_0^1(M,\mathcal E)$ with the norm
given by  (\ref{normHone}), i.e.,
\begin{equation}\label{NOTES_NORM_H_ONE}
\|F\|^2_{H_0^1(M,\mathcal E)}:=\inner{F,F}_2+\inner{\nabla^\tau F,\nabla^\tau F}_2.
\end{equation}
The advantage of using this norm, instead of the norm given by $\inner{\nabla^\tau F,\nabla^\tau F}_2$, is that we do not use the Poincaré constant,  and the estimates of the region of the resolvent set are obtained in term of explicit constants.
\end{remark}
\begin{remark}\label{LINFUNCCONST} The linear functional $\inner{f,G}_2$ is defined via a given $L^2$-function with values in the Clifford-Hilbert module, precisely it is
$$
\inner{f,G}_2:=\int_{M} \Sc{\overline{f}G} \: \dif\text{vol}.
$$
By Lemma \ref{lem_Properties} iv) and the definition of the scalar product we have
$$
|\inner{f,G}_2|=|{\rm Sc}\langle f,G\rangle|\leq \|f\|_2\|G\|_2
\leq  \|f\|_2\|G\|_{H_0^1}.
$$
\end{remark}

\begin{lemma}\label{LemmaDF64}
Let $D$ be the Dirac operator given locally by (\ref{Eq:General_Dirac}). Then, for all $F\in H^1(M,\mathcal E)$ and $G\in L^2(M,\mathcal E)$
\begin{equation}\label{ineqLemma}
|\langle DF,G \rangle_2| \leq C_n \|\nabla^{\tau} F\|_2 \, \| G\|_2, \ \ \ \ {\rm where} \ \ \ C_n:=\sqrt{n}.
\end{equation}
\end{lemma}
\begin{proof}
    By the Cauchy-Schwarz inequality we have
    $$
    |\langle DF,G \rangle_2| \leq \|DF\|_2 \|G\|_2,
    $$
    so it only remains to prove that
    $$
    \|DF\|_2 \leq \sqrt n \|\nabla^\tau F \|_2.
    $$
    Take a partition of unity $\varphi_\alpha$ subjected to an atlas $U_\alpha$. On each chart choose an orthonormal frame $E_1^\alpha,\dots E_n^\alpha$ and $e_1^\alpha,\dots,e_n^\alpha$ the corresponding sections of the Clifford bundle $Cl(TM)$
    \begin{align*}
    \|DF\|_2^2 &= \int_M \|DF(x)\|^2 \dif\text{vol} = \sum_{\alpha} \int_{U_\alpha} \varphi_\alpha(x) \|DF(x)\|^2 \dif\text{vol}
    \\
    &
    = \sum_{\alpha}\int_{U_\alpha } \varphi_\alpha(x) \|\sum_{i=1}^n e_i^\alpha(x) \nabla^\tau_{E_i^\alpha} F(x)\|^2 \dif\text{vol}.
    \end{align*}
    Let us consider a single term in the integrand. By triangle inequality we have
    $$
    \|\sum_{i=1}^n e_i^\alpha(x) \nabla^\tau_{E_i^\alpha} F(x)\| \leq \sum_{i=1}^n \| e_i^\alpha(x) \nabla^\tau_{E_i^\alpha} F(x)\|   \leq \sum_{i=1}^n\|e_i^\alpha(x)\| \|\nabla^\tau_{E_i^\alpha} F(x)\|,
    $$
    where in the last inequality $\|e_i^\alpha(x)\|$ stands for the norm of the multiplication operator by $e_i^\alpha(x)$. Let us compute this norm and prove that it equals $1$.\\ On a chart $U_\alpha$ we can consider a trivialization of the spin bundle $P_{spin}$ and of the Clifford bundle $Cl(TM)$. Therefore, we have constant sections $e_1,\dots,e_n$ of $Cl(TM)$. Since both a $e_i^\alpha(x)$ and $e_i$ are two orthonormal bases of $\BR^n \subset \BR_n$, there must exist a section $s_\alpha\in C^\infty(U_\alpha, P_{Spin}(TU_\alpha))$ such that
    $$
    e_i^\alpha(x) = \sigma(s_\alpha(x)^{-1})e_i=s_\alpha(x)^{-1}e_i s_\alpha (x), \qquad i=1,\dots,n.
    $$
    The section $s_\alpha$ is essentially a spin-valued function after we have chosen a trivialization, and we view the spin group as a subset of the Clifford algebra.   Recalling that $\overline{ab}=\bar b \bar a$ and $s_\alpha \overline{s_\alpha} =1$, i.e. $s_\alpha^{-1}=\overline{s_\alpha }$, for any $v\in \BR_n$ we have
    \begin{align*}
    \|e_i^\alpha(x) v\|^2 & = Sc(\overline{e_i^\alpha(x) v} \cdot e_i^\alpha(x) v)= Sc(\,\bar v\,\overline{e_i^\alpha(x)}\, e_i^\alpha(x)\, v)\\
    & = Sc(\bar v\overline {\sigma(s_\alpha(x)^{-1})e_i }\, \sigma(s_\alpha(x)^{-1})e_i\,  v)
    \\
    &
    =Sc (\overline v s_\alpha(x)^{-1} \bar e_i s_\alpha(x) s_\alpha(x)^{-1} e_i s_\alpha(x) v)
    \\
    &
    = Sc (\overline v v)
    \\
    & = \|v\|^2.
    \end{align*}
    Therefore, $\|e_i^\alpha(x)\| = 1$. We now use the discrete version of the Cauchy-Schwarz inequality $(\sum_{i=1}^n a_i)^2\leq n\sum_{i=1}^n a_i^2$, to compute

\begin{align*}
        \|DF\|_2^2 &\leq \sum_{\alpha}\int_{U_\alpha } \varphi_\alpha(x)  \left(\sum_{i=1}^n \| \nabla^\tau_{E_i^\alpha} F(x)\|\right)^2 \dif\text{vol}
        \\
        &\leq  \sum_{\alpha}\int_{U_\alpha } \varphi_\alpha(x)\,  n\sum_{i=1}^n \| \nabla^\tau_{E_i^\alpha} F(x)\|^2 \dif\text{vol}  \\
        &=  n\sum_{\alpha}\int_{U_\alpha }\varphi_\alpha(x)  \| \nabla^\tau F(x)\|^2 \dif\text{vol} \\
        &= n\int_M (\sum_{\alpha} \varphi_\alpha(x)) \| \nabla^\tau F(x)\|^2 \dif\text{vol} \\
        &= n\int_M \| \nabla^\tau F(x)\|^2 \dif\text{vol} = n\| \nabla^\tau F\|^2_2,
\end{align*}
where we have used $\sum_{i=1}^n \|\nabla^\tau_{E_i^\alpha} F(x)\|^2=\|\nabla^\tau F(x)\|^2$ and the fact that $\{\varphi_\alpha\}$ are a partition of unity. Taking the square root of the obtained inequality gives the assertion.
\end{proof}

We are now ready to state and prove the following main result:

\begin{theorem} \label{Th:Main_general}
Let $(M,g)$ be a compact spin manifold $M$ with boundary $\partial M$ and a Riemannian metric $g$ with scalar curvature $\kappa(x)$ and let
\begin{equation}\label{kappabar}
k_{min}:= \frac{1}{4}\min_{x\in M}\kappa(x).
\end{equation}
Let $s\in\BR^{n+1}$, and assume that
\begin{align}\label{eq:STIME_SU_ESSE_DIR}
|s|^2+k_{min}- |s_0|^2C_n^2>0.
\end{align} where $C_n$ is given by (\ref{ineqLemma}).
Then,
 for every $f\in L^2(M,\mathcal{E})$ there exists a unique $F_f\in H_0^1(M,\mathcal{E})$ such that
		\begin{equation*}
			q_s(F_f,G) = \inner{f,G}_{2}, \qquad \text{for all } G\in H_0^1(M,\mathcal{E}).
		\end{equation*}
Moreover, for $s\in \mathbb{R}^{n+1}$ that satisfies estimate (\ref{eq:STIME_SU_ESSE_DIR}), we have
	\begin{equation}\label{quatSresorig_DIR_EST_new}
\|S_R^{-1}(s,D)f\|_2\leq  \frac{1}{\sqrt{|s|^2+k_{\rm min}}-|s_0|C_n}\left(C_n+\frac{|s|}{\sqrt{|s|^2+k_{\rm min}}}\right)
\norm{f}_2.
		\end{equation}

	\end{theorem}

	\begin{proof}
		We will verify that the bilinear form $q_s$ satisfies the assumptions of Lax-Milgram lemma.
	First of all, the continuity of the bilinear form $q_s$ defined in (\ref{EQ:Form_gen_inner}) follows from the estimates
\begin{equation} \label{EQ:Form_gen_inner2}
			|q_s(F,G)|\leq  |\inner{\nabla^\tau F,\nabla^\tau G}_2| + \frac{1}{4}|\inner{\kappa F,G}_2| + 2|s_0| |\inner{DF,G}_2| + \modulo{s}^2|\inner{F,G}_2|.
		\end{equation}
		Using Lemma \ref{LemmaDF64}, and the fact that $\kappa$ is a smooth real valued function on a compact manifold,
and therefore attains minimum $\kappa_{\rm min}$ and maximum $\kappa_{\rm max}$, we have that
$$
\frac 14|\kappa(x)|\leq k_m:=\max(\frac 14|\kappa_{\rm min}|, \frac 14|\kappa_{\rm max}|),
$$
 we have
		\begin{align*}
			\modulo{q_s(F,G)}  &\leq \|\nabla^{\tau} F\|_2 \, \| \nabla^\tau G\|_2 +
k_m\norm{F}_2\norm{G}_2
\\
&+ 2\modulo{s_0}C_n\, \norm{\nabla^\tau F}_2\norm{G}_2 + \modulo{s}^2\norm{F}_2\norm{G}_2.
\end{align*}
Therefore,  the form
$$
q_s(\cdot,\cdot):H_0^1(M,\mathcal{E})\times H_0^1(M,\mathcal{E}) \to \mathbb{R}
$$
is continuous since
there exists a positive constant $C(\modulo{s},\modulo{s_0},\kappa, M, n,\mu)$ such that
\begin{align*}
|q_s(F,G)|  \leq  C(\modulo{s},\modulo{s_0}, k_m,n)\norm{F}_{H_0^1}\norm{G}_{H_0^1},
\ \ {\rm for\ all}\ \  F,G\in H_0^1(M,\mathcal{E}).	
	\end{align*}
Let us consider  the coercivity of $q_s$. Setting $F=G$ in the form $q_s(F,G)$ and using
Lemma \ref{LemmaDF64}, we have
\begin{align} \label{Eq:Coer_General_nonneg}
			q_s(F,F) & = \inner{\nabla^\tau F,\nabla^\tau F}_2  + \frac{1}{4} \inner{\kappa  F,F}_2 - 2s_0\inner{DF ,F}_2
  + |s|^2\inner{F ,F}_2  \nonumber \\ & \geq \norm{\nabla^\tau F}_2^2 + k_{\rm min}\norm{F}_2^2 - 2|s_0|C_n\, \norm{\nabla^\tau F}_2\norm{F}_2 + |s|^2\norm{F}_2^2,
		\end{align}
where we have used that
$$
\int_M \Sc{\overline{\kappa F}F} dvol \geq \min_{x\in M}{\kappa(x)} \int_M \Sc{\overline{F}F} dvol
$$
 and $k_{\rm min}=\frac 14 \min_{x\in M} \kappa(x)$. Using Young's inequality, for some $\delta >0$, we further obtain
		\begin{equation*}
			\norm{\nabla^\tau F}_2 (|s_0|\norm{F}_2) \leq \frac{1}{2\delta}\norm{\nabla^\tau F}_2^2 + \frac{\delta}{2}|s_0|^2\norm{F}_2^2,
		\end{equation*}
so that we can estimate the lower bound for \eqref{Eq:Coer_General_nonneg} by
		\begin{align} \label{Eq:lower_bound_general_nonneg}
			q_s(F,F) & \geq \norm{\nabla^\tau F}_2^2 + k_{min}\norm{F}_2^2
- 2C_n\p{\frac{1}{2\delta}\norm{\nabla^\tau F}_2^2
+ \frac{\delta}{2}|s_0|^2\norm{F}_2^2} + |s|^2\norm{F}_2^2 \nonumber
\\ &
 \geq \Big(1-\frac{C_n}{\delta}\Big)
 \norm{\nabla^\tau F}_2^2 + \p{|s|^2 +k_{min}  - |s_0|^2
 \delta C_n}\norm{F}_2^2.
		\end{align}
The coercivity requires that two coefficients into parenthesis in (\ref{Eq:lower_bound_general_nonneg}) must be positive and so
\begin{equation}\label{AandB_DELTA}
A(\delta):=1-\frac{C_n}{\delta}>0, \qquad B(\delta):=|s|^2+k_{min}  - |s_0|^2\delta C_n >0.
\end{equation}
Let us check when such a $\delta$ exists. If $s_0 = 0$, then the only constraint on $\delta$ is given by $A(\delta)>0$, which is satisfied for every $\delta> C_n$. In addition, we must have $|s|^2 + k_{\min}>0$, which coincides with condition~\eqref{eq:STIME_SU_ESSE_DIR} in this case.

Next assume that $s_0 \neq 0$. Then the two inequalities \eqref{AandB_DELTA} yield
\begin{equation}
    \label{eq:range_for_delta_dirichlet}
    C_n <\delta < \frac{|s|^2 + k_{\min}}{ C_n|s_0|^2}.
\end{equation}
Such a $\delta$ will exist if and only if
$$
C_n < \frac{|s|^2 + k_{\min}}{ C_n|s_0|^2} \iff |s|^2+k_{\rm min}  - |s_0|^2 C_n^2 >0
$$
which gives exactly~\eqref{eq:STIME_SU_ESSE_DIR}.

 Hence, we have proven that $q_s$ is coercive in $H_0^1(M,\mathcal{E})$ when $s\in \mathbb{R}^{n+1}$ satisfies the inequality (\ref{eq:STIME_SU_ESSE_DIR}).
 Fixing now any function $f\in L^2(M,\mathcal{E})$, we can consider the corresponding functional
		\begin{equation*}
			\varphi_f(G) := \inner{f,G}_2, \quad G\in H_0^1(M,\mathcal{E})
		\end{equation*}
as in Problem \ref{PROB_DIR_FULL_NORM_DIR}.
		Then by Remark \ref{LINFUNCCONST}, this functional is bounded in $H_0^1(M,\mathcal{E})$, in fact
		\begin{equation*}
			\modulo{\varphi_f(G)} = \modulo{\inner{f,G}_{2}} \leq \norm{f}_2\norm{G}_2 \leq \norm{f}_2\norm{G}_{H_0^1}, \quad G\in H_0^1(M,\mathcal{E}).
		\end{equation*}
		Hence, the assumptions of Lax Milgram lemma are satisfied and so there exists a unique weak solution $F_f\in H_0^1(M,\mathcal{E})$ such that
		\begin{equation} \label{Eq:test_funct_general_nonneg}
			q_s(F_f,G) = \varphi_f(G) = \inner{f,G}_2, \quad \text{for all } G\in H_0^1(M,\mathcal{E}).
		\end{equation}
Our next task is to prove \eqref{quatSresorig_DIR_EST_new}.
To this end we proceed in various steps. \\
Step 1. First of all we take a fixed $\delta$ in the range \eqref{eq:range_for_delta_dirichlet}, then we test (\ref{Eq:test_funct_general_nonneg}) with $G = F_f$.
Still considering  Problem \ref{PROB_DIR_FULL_NORM_DIR} and using also Lemma \ref{lem_Properties} iv) we have the estimate
		\begin{equation} \label{Eq:norm_sc_bound_general_nonneg}
			|q_s(F_f,F_f)| = |\inner{f,F_f}_2| \leq \norm{f}_2\norm{F_f}_2,
		\end{equation}
and moreover, using \eqref{Eq:lower_bound_general_nonneg} and \eqref{AandB_DELTA} we get
        \begin{equation}
            \label{eq:coercivity_bounded_dirichlet}
            A(\delta) \norm{\nabla^\tau F_f}_2^2 + B(\delta) \norm{F_f}_2^2 \leq q_s(F_f,F_f) \leq \norm{f}_2\norm{F_f}_2.
        \end{equation}
        Thus we deduce
\begin{equation}
    \label{eq:norm_estimate_dirichlet}
    \norm{F_f}_2 \leq \frac{\norm{f}_2}{B(\delta)},
\end{equation}
        and from the same inequality we also have
        $$
        A(\delta) \norm{\nabla^\tau F_f}_2^2 \leq \norm{f}_2\norm{F_f}_2.
        $$
Combining this with~\eqref{eq:norm_estimate_dirichlet} we get
\begin{equation}
    \label{eq:grad_norm_estimate_dirichlet}
    \norm{\nabla^\tau F_f}_2 \leq \frac{\norm{f}_2}{\sqrt{A(\delta)B(\delta)}}.
\end{equation}
Step 2.
Note that the inequality~\eqref{eq:coercivity_bounded_dirichlet} is actually a family of inequalities depending on $\delta$ in the range~\eqref{eq:range_for_delta_dirichlet} which we can denote as
\[
  I \;:=\; \left(C_n,\;\frac{|s|^2+k_{\min}}{|s_0|^2 C_n}\right).
\]
 Therefore, to optimize the quantities in~\eqref{eq:norm_estimate_dirichlet} and \eqref{eq:grad_norm_estimate_dirichlet}, we fix the parameters $C_n>0$, $s_0\in\mathbb{R}$, $s\in\mathbb{R}^{n+1}$,
and $k_{\min}\in\mathbb{R}$ with
\[
  |s|^2+k_{\min}-|s_0|^2 C_n^{\,2}>0.
\]

Since from the solution of the Dirichlet boundary conditions problem we
have
$
F_f = Q_s(D)^{-1}f,
$
then the equations above give us
$$
\norm{ Q_s(D)^{-1}f}_2\leq\frac{\norm{f}_2}{B(\delta)}, \ \ \  \ \ \
\norm{\nabla^\tau Q_s(D)^{-1}f}_2 \leq  \frac{\norm{f}_2}{\sqrt{A(\delta)B(\delta)}},
$$
where the Dirac operator $D=\sum_{i=1}^ne_i\nabla^\tau_{E_i}$ is given by
 (\ref{Eq:General_Dirac}) so the $S$-resolvent is
\begin{equation}\label{quatSresorig_DIR}
S_R^{-1}(s,D):=-(D-\overline{s}\mathcal{I})Q_s(D)^{-1}
=-(\sum_{i=1}^ne_i\nabla^\tau_{E_i}-\overline{s}\mathcal{I})Q_s(D)^{-1}
\end{equation}
and taking the norm, and using Lemma \ref{LemmaDF64}, we obtain
\begin{align*}
\|S_R^{-1}(s,D)f\|_2&\leq
\|(\sum_{i=1}^ne_i\nabla^\tau_{E_i})Q_s(D)^{-1}f\|_2+\|\overline{s}Q_s(D)^{-1}f\|_2
\\
&
\leq\Big(
 \frac{ C_n}{\sqrt{A(\delta)B(\delta)}}
 +\frac{  |s|}{B(\delta)}\Big)
\norm{f}_2.
\end{align*}

Step 3.
Let us set
$$
G(\delta):=\frac{ C_n}{\sqrt{A(\delta)B(\delta)}}
 +\frac{  |s|}{B(\delta)}
$$
and let us find suitable bounds for it.
We note that for $\delta \to C_n^+$ and $\delta \to \left(\frac{|s|^2+k_{\min}}{|s_0|^2 C_n}\right)^-$ the function $G$ tends to $+\infty$.
Since $G$ is continuous on the open interval $I$ and tends to $+\infty$ at
both endpoints, it admits a global minimum on $I$, attained at an
interior point $\tilde{\delta}\in I$.  The point $\tilde\delta$ can be, in principle, computed by studying $G'(\delta)$. Standard but cumbersome computations show that  $G'(\delta)=0$ gives an algebraic equation of degree $5$ whose roots are not immediate to compute so we proceed to provide bounds for $\tilde\delta$.
\\
We observe that, since $G(\delta)$ is sum of two positive quantities, we have that
$$
\min_{\delta\in I}G(\delta)\geq \inf_{\delta\in I}\frac{ C_n}{\sqrt{A(\delta)B(\delta)}} + \inf_{\delta\in I} \frac{  |s|}{B(\delta)}= \frac{ C_n}{\alpha_2}+ \frac{  |s|}{\alpha_1},
$$
where $\alpha_1=\sup_{\delta\in I} B(\delta)$ and  $\alpha_2=\sup_{\delta\in I} \sqrt{A(\delta)B(\delta)}$.
 Since $B(\delta)$ is linear and decreasing in $\delta$, we get the supremum for $\delta$ tending to the left endpoint, namely:
$$
\alpha_1 = \modulo{s}^2  + k_{min} - \modulo{s_0}^2 C_n^2 .
$$
To determine $\alpha_2$ we note that $A(\delta)$ vanishes on the left limit of $\delta$ in~\eqref{eq:range_for_delta_dirichlet} and $B(\delta)$ on the right one. In addition, both functions are positive and smooth in the interval $I$. Therefore, there must be at least one maximum inside $I$. It is straightforward to check that maximum is attained at
$$
\delta^* = \frac{\sqrt{|s|^2+k_{\rm min}}}{|s_0|}.
$$
Plugging this into $A(\delta)B(\delta)$ gives
$$
\alpha_2 = \sqrt{|s|^2+k_{\rm min}}-|s_0|C_n.
$$

Thus we have the lower bound
$$
\min_{\delta\in I}G(\delta)> \frac{ C_n}{\sqrt{|s|^2+k_{\rm min}}-|s_0|C_n}+ \frac{  |s|}{\modulo{s}^2  + k_{\rm min} - \modulo{s_0}^2 C_n^2}.
$$

Moreover, since the point $\tilde\delta$ gives the minimum of $G(\delta)$, it is immediate that
$$
G(\tilde\delta)\leq G(\delta^*)=G(\frac{\sqrt{|s|^2+k_{\rm min}}}{|s_0|})
=\frac{C_n}{\sqrt{|s|^2+k_{\rm min}}-|s_0|C_n}+\frac{|s|}{|s|^2+k_{\rm min}- |s_0|C_n\sqrt{|s|^2+k_{\rm min}}}
$$
$$
=\frac{1}{\sqrt{|s|^2+k_{\rm min}}-|s_0|C_n}\left(C_n+\frac{|s|}{\sqrt{|s|^2+k_{\rm min}}}\right).
$$
To summarize the minimum $G(\tilde\delta)$ satisfies the estimates
$$
\frac{ C_n}{\sqrt{|s|^2+k_{\rm min}}-|s_0|C_n}+ \frac{  |s|}{\modulo{s}^2  + k_{\rm min} - \modulo{s_0}^2 C_n^2}
< G(\tilde\delta)
$$
$$
\leq\frac{1}{\sqrt{|s|^2+k_{\rm min}}-|s_0|C_n}\left(C_n+\frac{|s|}{\sqrt{|s|^2+k_{\rm min}}}\right).
$$
The right-hand side of these inequalities gives the bound in the assertion.
\end{proof}

\begin{remark}
The $S$-resolvent depends significantly on the sign of $k_{min}$. We have~\eqref{eq:STIME_SU_ESSE_DIR} equivalent to
$$
|s|^2+k_{min}- |s_0|^2C_n^2>0.
$$
Letting $s = s_0 + s_1J$, the condition $|s|^2 = s_0^2+s_1^2$ leads to
\begin{equation}
    \label{eq:resolvent_sets_condition}
    s_1^2 - (C_n^2-1)s_0^2 > -k_{min}.
\end{equation}
For the geometric picture below since $C_{n}=\sqrt{n}>1$ so that $C_{n}^{2}-1>0$ and the level sets of~\eqref{eq:resolvent_sets_condition} are hyperbolae.
In particular, there is a spectral gap when $k_{\min}>0$.  Further improvement on the resolvent set can be achieved by using Poincar\'e inequality.
\end{remark}

\medskip
\noindent
{\bf Discussion of the various cases.}

{\em Case 1.}
The assumption that the minimum scalar curvature is strictly positive, i.e., $k_{min}>0$, imposes a geometric restriction on the spectral parameter $s$ for the $S$-resolvent operator. This restriction is defined by the inequality:
\begin{equation}\label{CONWITHK}
(C_n^2-1)s_0^2-s_1^2<k_{min}.
\end{equation}

This inequality defines the region of existence in the complex plane $\mathbb{C}_J$ (where $s=s_0 + s_1 J$ for $J \in \mathbb{S}$).  Specifically, it represents the connected region, containing the origin bounded by the hyperbola of equation
$$
(C_n^2-1)s_0^2-s_1^2=k_{min},
$$
opening along the $s_{0}$-axis.
Since $C_n^2-1 > 0$ (recall the standing assumption $C_{n}>1$), the hyperbola intersects the $s_0$-axis of the complex plane $\mathbb{C}_J$ at the points:
$$
s_0=\pm\sqrt{\frac{k_{min}}{C_n^2-1}}.
$$
The origin lies in the resolvent set is visible directly: at $s=0$, the left-hand side of~\eqref{CONWITHK} equals $0<k_{\min}$.
This means that the spectral gap, which is the region where the $S$-resolvent is defined, satisfies the following condition on the real part of the spectral parameter:
$$
|\mathrm{Re}(s)| \,<\, \sqrt{\frac{k_{min}}{C_n^2-1}}.
$$

This inequality thus provides a quantitative bound for the spectral gap when the scalar curvature is strictly positive.

{\em Case 2.} When $k_{\min}:= \frac{1}{4}\min_{x\in M}\kappa(x) = 0$, we have a simplified condition for the existence of the $S$-resolvent, given by
\begin{equation}\label{STIME_SU_ESSE_DIR_NEW}
|s|^2 - |s_0|^2 C_n^2 > 0,
\end{equation}
which imposes a geometric restriction on the spectral parameter $s$.
Letting $s = s_0 + s_1 J$, the condition $|s|^2 = s_0^2 + s_1^2$ leads to:
\[
s_0^2 + s_1^2 - s_0^2 C_n^2 > 0.
\]
Since $C_n^2 - 1 > 0$, this inequality defines the region of existence in the complex plane $\mathbb{C}_J$ by:
\[
s_1^2 > (C_n^2 - 1)s_0^2.
\]
The boundary of this region is defined by the equation:
\[
s_1^2 - (C_n^2 - 1)s_0^2 = 0,
\]
which represents a pair of lines passing through the origin with slopes $\pm\sqrt{C_{n}^{2}-1}$. These lines partition the plane into four open wedges; the resolvent set consists of the two wedges
\[
\bigl\{|s_{1}|>\sqrt{C_{n}^{2}-1}\,|s_{0}|\bigr\}
\]
containing the $s_{1}$-axis (above and below). Strictly speaking, the resolvent set is therefore disconnected; the origin lies on its boundary and does not belong to it.
In particular, in this case there is no spectral gap.

{\em Case 3.}
The case where the minimum scalar curvature is negative, i.e., $k_{min} < 0$, requires a different restriction on the spectral parameter $s$. By setting $a^2 := -k_{min}$, the condition for the $S$-resolvent to exist is defined by the inequality:
\begin{equation}\label{CONWITHKNEG}
s_1^2-(C_n^2-1)s_0^2 > a^2.
\end{equation}
Since $C_n^2-1 > 0$, the hyperbola
$$
s_1^2-(C_n^2-1)s_0^2=a^2
$$
intersects the imaginary axis ($s_0=0$) at the points $s_1=\pm a$.
The hyperbola opens along the $s_{1}$-axis and consists of two branches: an upper one, $\{s_{1}\geq+\sqrt{a^{2}+(C_{n}^{2}-1)s_{0}^{2}}\}$, and a lower one, $\{s_{1}\leq -\sqrt{a^{2}+(C_{n}^{2}-1)s_{0}^{2}}\}$.  The resolvent set $\rho_{S}=\{(s_{0},s_{1})\colon s_{1}^{2}-(C_{n}^{2}-1)s_{0}^{2}>a^{2}\}$ is the union of the two open regions above the upper branch and below the lower branch; it is therefore disconnected and does not contain the origin (which lies in the $S$-spectrum, since the inequality fails at $(0,0)$).
Also in this case there is no spectral gap.

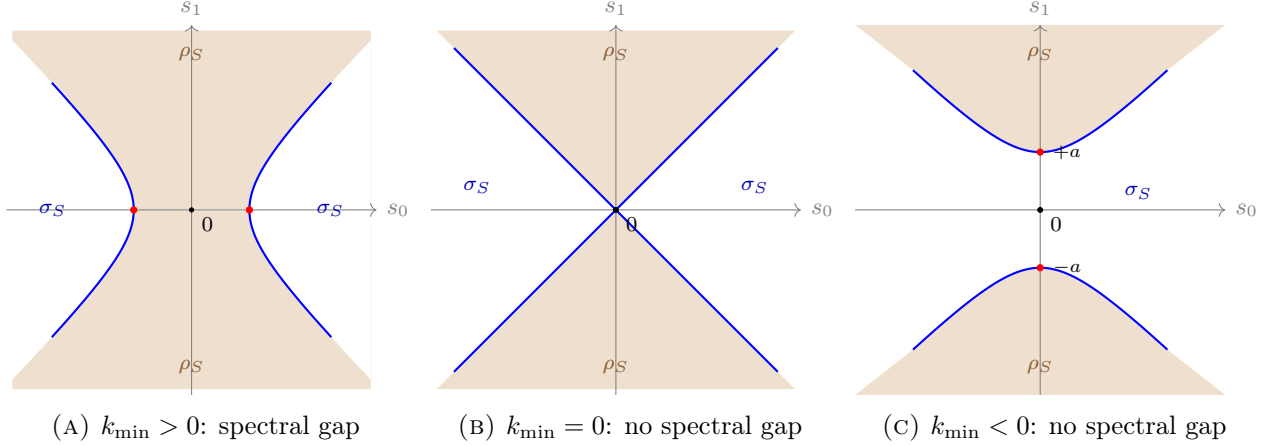
\begin{figure}[htbp]
\centering
\begin{subfigure}[b]{0.32\textwidth}
\centering
\begin{tikzpicture}[scale=0.765]
\fill[brown!25]
   (-3.1, 3.1) -- (3.1, 3.1) -- (3.1, 0)
   -- plot[domain=3.1:1, samples=80] (\x, {sqrt(\x*\x - 1)})
   -- (-1, 0)
   -- plot[domain=-1:-3.1, samples=80] (\x, {sqrt(\x*\x - 1)})
   -- (-3.1, 0) -- cycle;
\fill[brown!25]
   (-3.1, -3.1) -- (3.1, -3.1) -- (3.1, 0)
   -- plot[domain=3.1:1, samples=80] (\x, {-sqrt(\x*\x - 1)})
   -- (-1, 0)
   -- plot[domain=-1:-3.1, samples=80] (\x, {-sqrt(\x*\x - 1)})
   -- (-3.1, 0) -- cycle;
\draw[->, gray] (-3.2,0) -- (3.2,0) node[right] {\footnotesize $s_0$};
\draw[->, gray] (0,-3.2) -- (0,3.2) node[above] {\footnotesize $s_1$};
\draw[thick, blue, domain=-2.2:2.2, samples=80] plot ({sqrt(1+\x*\x)}, \x);
\draw[thick, blue, domain=-2.2:2.2, samples=80] plot ({-sqrt(1+\x*\x)}, \x);
\filldraw[red] (1,0) circle (1.5pt);
\filldraw[red] (-1,0) circle (1.5pt);
\node[below, font=\tiny] at (1,-0.1) {};
\node[below, font=\tiny] at (-1,-0.1) {};
\filldraw[black] (0,0) circle (1pt);
\node[below right, font=\scriptsize] at (0,0.05) {$0$};
\node[font=\scriptsize, blue!70!black] at (2.4,0) {$\sigma_S$};
\node[font=\scriptsize, blue!70!black] at (-2.4,0) {$\sigma_S$};
\node[font=\scriptsize, brown!70!black] at (0,2.7) {$\rho_S$};
\node[font=\scriptsize, brown!70!black] at (0,-2.7) {$\rho_S$};
\end{tikzpicture}
\caption{$k_{\min}>0$: spectral gap}
\label{fig:first}
\end{subfigure}
\hfill
\begin{subfigure}[b]{0.32\textwidth}
\centering
\begin{tikzpicture}[scale=0.765]
\fill[brown!25] (0,0) -- (-3.1,3.1) -- (3.1,3.1) -- cycle;
\fill[brown!25] (0,0) -- (-3.1,-3.1) -- (3.1,-3.1) -- cycle;
\draw[->, gray] (-3.2,0) -- (3.2,0) node[right] {\footnotesize $s_0$};
\draw[->, gray] (0,-3.2) -- (0,3.2) node[above] {\footnotesize $s_1$};
\draw[thick, blue] (-2.8,-2.8) -- (2.8,2.8);
\draw[thick, blue] (-2.8,2.8) -- (2.8,-2.8);
\filldraw[black] (0,0) circle (1.2pt);
\node[below right, font=\scriptsize] at (0,0.05) {$0$};
\node[font=\scriptsize, brown!70!black] at (0, 2.7) {$\rho_S$};
\node[font=\scriptsize, brown!70!black] at (0, -2.7) {$\rho_S$};
\node[font=\scriptsize, blue!70!black] at (2.4, 0.4) {$\sigma_S$};
\node[font=\scriptsize, blue!70!black] at (-2.4, 0.4) {$\sigma_S$};
\end{tikzpicture}
\caption{$k_{\min}=0$: no spectral gap}
\label{fig:second}
\end{subfigure}
\hfill
\begin{subfigure}[b]{0.32\textwidth}
\centering
\begin{tikzpicture}[scale=0.765]
\fill[brown!25]
  (-3.2, 3.2) -- (3.2, 3.2) --
  plot[domain=2.2:-2.2, samples=80] (\x, {sqrt(1+\x*\x)}) -- cycle;
\fill[brown!25]
  (-3.2, -3.2) -- (3.2, -3.2) --
  plot[domain=2.2:-2.2, samples=80] (\x, {-sqrt(1+\x*\x)}) -- cycle;
\draw[->, gray] (-3.2,0) -- (3.2,0) node[right] {\footnotesize $s_0$};
\draw[->, gray] (0,-3.2) -- (0,3.2) node[above] {\footnotesize $s_1$};
\draw[thick, blue, domain=-2.2:2.2, samples=80] plot (\x, {sqrt(1+\x*\x)});
\draw[thick, blue, domain=-2.2:2.2, samples=80] plot (\x, {-sqrt(1+\x*\x)});
\filldraw[red] (0,1) circle (1.5pt);
\filldraw[red] (0,-1) circle (1.5pt);
\node[right, font=\tiny] at (0.05,1) {$+a$};
\node[right, font=\tiny] at (0.05,-1) {$-a$};
\filldraw[black] (0,0) circle (1.2pt);
\node[below right, font=\scriptsize] at (0,0.05) {$0$};
\node[font=\scriptsize, blue!70!black] at (1.7, 0.3) {$\sigma_S$};
\node[font=\scriptsize, brown!70!black] at (0, 2.7) {$\rho_S$};
\node[font=\scriptsize, brown!70!black] at (0, -2.7) {$\rho_S$};
\end{tikzpicture}
\caption{$k_{\min}<0$: no spectral gap}
\label{fig:third}
\end{subfigure}
\caption{Part of the $S$-resolvent set $\rho_{S}$ (shaded brown) and of the $S$-spectrum $\sigma_{S}$ (unshaded), determined by $s_{1}^{2}-(C_{n}^{2}-1)s_{0}^{2}>-k_{\min}$ with $C_{n}>1$, hyperbolae and asymptotic lines are in blue, vertices (when present) in red.
A spectral gap is an open interval of the real axis $\{s_{1}=0\}$ contained entirely in $\rho_{S}$, separating the $S$-spectrum on that axis into two disjoint unbounded pieces.
Such a gap exists only when $k_{\min}>0$, figure~(A), the spectrum on the real axis is then $|s_{0}|\geq \sqrt{k_{\min}/(C_{n}^{2}-1)}$, the two rays being separated by $(-\sqrt{k_{\min}/(C_{n}^{2}-1)}, \sqrt{k_{\min}/(C_{n}^{2}-1)})$, of length $2\sqrt{k_{\min}/(C_{n}^{2}-1)}$. When $k_{\min}=0$, see figure (B), the gap collapses to the origin, when $k_{\min}<0$, see (C), the whole real axis is spectrum, so no gap exists. In Figure (A) the brown region is split only for visual analogy with (B) and (C), there $\rho_{S}$ is actually connected, the two pieces joined through the gap along the $s_{0}$-axis.
}
\label{fig:resolvent_sets}
\end{figure}

\section{The $S$-resolvent estimates using Poincar\'e Inequality}\label{POINCARESECT}

In this section we establish the $S$-resolvent
estimates under
homogeneous Dirichlet boundary conditions using $\norm{\nabla^\tau F}_2$ as a norm in $H^1_0$.
In this case we obtain the $S$-resolvent estimates that are subjected to bounds that involve the Poincar\'e constant.

\begin{theorem}[Poincar\'e]
Let $(M,g)$ be a compact connected spin manifold  with $\partial M\neq\emptyset$.
Then, there exists a positive constant $C_P(M)$ such that
\begin{equation}\label{PoincareDD}
\norm{F}_2 \leq C_P(M)\norm{\nabla^\tau F}_2,\ \ \ \ {\rm for\ every} \ \ F\in H^1_0(M,\mathcal{E}).
\end{equation}
\end{theorem}
Using Poincar\'e Inequality, for the manifolds for which~\eqref{PoincareDD} holds on $H^1_0$, we consider the norm
$$
\|F\|_{H^1_0}:=\norm{\nabla^\tau F}_2.
$$
\begin{theorem}[Nonnegative curvature using Poincar\'e Inequality] \label{Th:Main_general_nonneg_POIN}
Let $(M,g)$ be a compact spin manifold $M$ with boundary $\partial M$ and $g$ be a Riemannian metric. Let $\kappa$ be the scalar curvature and let us assume that
\begin{equation}\label{kappabar2}
k_{min}:= \frac{1}{4}\min_{x\in M}\kappa(x) \geq 0 .
\end{equation}
Let $s\in\BR^{n+1}$, and assume that
\begin{align}\label{STIME_SU_ESSE_DIR_POIN}
1-2\modulo{s_0}C_n C_P(M)>0,
\end{align}
where $C_n$ is given by (\ref{ineqLemma}) and $C_P(M)$ is the Poincar\'e constant.
Then, for every $f\in L^2(M,\mathcal{E})$ there exists a unique $F_f\in H_0^1(M,\mathcal{E})$ such that
\begin{equation*}
q_s(F_f,G) = \inner{f,G}_{2}, \qquad \text{for all } G\in H_0^1(M,\mathcal{E}).
\end{equation*}
Moreover, for $s\in \mathbb{R}^{n+1}$ that satisfies estimate (\ref{STIME_SU_ESSE_DIR_POIN}), we have
\begin{equation}\label{quatSresorig_DIR_EST_POIN}
\|S_R^{-1}(s,D)f\|_2 \leq  C_P(M)\frac{C_n+C_P(M)|s| }{1-2\modulo{s_0}C_n C_P(M)}\norm{f}_2.
\end{equation}
\end{theorem}

\begin{proof}
We just consider the coercivity condition; the continuity condition is standard.

We observe that (assuming $F=G$ in the inner product of $q_s(F,G)$), and recalling from the previous proof that $q_{s}(F,F)\in\mathbb{R}$ since $\kappa$ is real-valued and the Dirac operator $D$ is symmetric:
\begin{align} \label{Eq:Coer_General_POINCARE}
    q_s(F,F) &= \inner{\nabla^\tau F,\nabla^\tau F}_2 + \frac{1}{4} \inner{\kappa F,F}_2 - 2s_0\inner{DF ,F}_2 + |s|^2\norm{F}_2^2 \nonumber \\
    &\geq \norm{\nabla^\tau F}_2^2 +k_{\min}\norm{F}_2^2 - 2|s_0|C_n\, \norm{\nabla^\tau F}_2\norm{F}_2 + |s|^2\norm{F}_2^2.
\end{align}
We apply the Poincar\'e inequality~\eqref{PoincareDD} to the term $2|s_0|C_n\, \norm{\nabla^\tau F}_2\norm{F}_2$:
\[
   2|s_{0}|\,C_{n}\,\|\nabla^{\tau}F\|_{2}\,\|F\|_{2} \;\leq\; 2|s_{0}|\,C_{n}\,\|\nabla^{\tau}F\|_{2}\cdot C_{P}(M)\,\|\nabla^{\tau}F\|_{2} \;=\; 2|s_{0}|\,C_{n}\,C_{P}(M)\,\|\nabla^{\tau}F\|_{2}^{2}.
\]
Inserting this bound into~\eqref{Eq:Coer_General_POINCARE},
\begin{align*}
    q_s(F,F) &\geq \norm{\nabla^\tau F}_2^2 \p{1 - 2|s_0|C_n C_P(M)} + \norm{F}_2^2 \p{k_{\min} + |s|^2} \\
    &\geq \norm{\nabla^\tau F}_2^2 \p{1 - 2|s_0|C_n C_P(M)},
\end{align*}
where the last inequality drops the non-negative term $(k_{\min}+|s|^{2})\|F\|_{2}^{2}\geq 0$ (using $k_{\min}\geq 0$ from~\eqref{kappabar2}). This makes the resulting bound non-sharp, as noted in Remark~\ref{rem:non_optimal} below.

We require the coefficient of $\norm{\nabla^\tau F}_2^2$ to be strictly positive:
\[
1-2|s_0|C_n C_P(M)>0.
\]
Under this condition, $q_s$ is coercive on $H_{0}^{1}(M,\mathcal{E})$ endowed with the gradient norm $\|F\|_{H_{0}^{1}}=\|\nabla^{\tau}F\|_{2}$ (which is equivalent to the full $H^{1}$ norm by~\eqref{PoincareDD}):
\begin{equation}\label{POINCAREGGG}
    q_s(F,F) \geq \Big(1- 2|s_0|C_n C_P(M) \Big)\norm{\nabla^\tau F}_2^2.
\end{equation}
So we can derive estimates for the solution $F_f$.
The Lax--Milgram lemma implies the existence of a unique solution $F_f$. Testing the defining equation $q_s(F_f,G) = \inner{f,G}_2$ with $G=F_f$:
\begin{equation}\label{ESTIM_PSEUDO_RESOL_DR_POIN}
    \p{1 - 2|s_0|C_n C_P(M)}\norm{\nabla^\tau F_f}_2^2 \leq q_s(F_f,F_f) = \inner{f,F_f}_{2} \leq |\inner{f,F_f}_{2}|.
\end{equation}
Here the equality is the defining relation~\eqref{Eq:test_funct_general_nonneg} of~$F_{f}$ tested with $G=F_{f}$; since $q_{s}(F_{f},F_{f})$ is real and non-negative (by~\eqref{POINCAREGGG} under~\eqref{STIME_SU_ESSE_DIR_POIN}), the equation $q_{s}(F_{f},F_{f})=\inner{f,F_{f}}_{2}$ forces $\inner{f,F_{f}}_{2}$ to be real and non-negative as well, hence $\inner{f,F_{f}}_{2}=|\inner{f,F_{f}}_{2}|. $

Applying the Cauchy--Schwarz inequality, $|\inner{f,F_{f}}_{2}|\leq\|f\|_{2}\|F_{f}\|_{2}$, and then the Poincar\'e inequality
\[
\norm{F_f}_2 \leq C_P(M)\norm{\nabla^\tau F_f}_2,
\]
we have
\begin{equation}\label{ESTIM_PSEUDO_RESOL_DR_POIN_BIS}
    \p{1 - 2|s_0|C_n C_P(M)}\norm{\nabla^\tau F_f}_2^2 \leq \norm{f}_2\norm{F_f}_2 \leq \norm{f}_2 C_P(M)\norm{\nabla^\tau F_f}_2.
\end{equation}
Dividing by $\norm{\nabla^\tau F_f}_2$ (assuming $F_f \neq 0$, since the case $F_f=0$ is trivial) gives the gradient estimate:
\begin{equation}
    \norm{\nabla^\tau F_f}_2 \leq \frac{C_P(M)}{1 - 2|s_0|C_n C_P(M)}\norm{f}_2.
\end{equation}
Using the Poincar\'e inequality again, we get the $L^2$-estimate:
\begin{equation}\label{Poincare_BIS}
    \norm{F_f}_2 \leq C_P(M)\norm{\nabla^\tau F_f}_2 \leq \frac{C_P(M)^2}{1 - 2|s_0|C_n C_P(M)}\norm{f}_2.
\end{equation}

Finally we get the $S$-resolvent estimate.
Recalling that $F_f = Q_s(D)^{-1}f$ and the $S$-resolvent is defined as $S_R^{-1}(s,D)f = -(D-\overline{s}\mathcal{I})F_f$, we take the norm:
\begin{align*}
\|S_R^{-1}(s,D)f\|_2 &\leq \|DF_f\|_2 + |\overline{s}|\norm{F_f}_2 \\
&\leq C_n \norm{\nabla^\tau F_f}_2 + |s|\norm{F_f}_2 \quad \text{(using } \norm{DF}_2 \le C_n \norm{\nabla^\tau F}_2).
\end{align*}
Substituting the derived bounds for $\norm{\nabla^\tau F_f}_2$ and $\norm{F_f}_2$:
\begin{align*}
\|S_R^{-1}(s,D)f\|_2 &\leq C_n \p{\frac{C_P(M)}{1 - 2|s_0|C_n C_P(M)}\norm{f}_2} + |s|\p{\frac{C_P(M)^2}{1 - 2|s_0|C_n C_P(M)}\norm{f}_2} \\
&= \frac{C_P(M)(C_n + C_P(M)|s|)}{1 - 2|s_0|C_n C_P(M)}\norm{f}_2.
\end{align*}
\end{proof}

\begin{remark}\label{rem:non_optimal}
Observe that the spectral estimate from~\eqref{STIME_SU_ESSE_DIR_POIN}
\begin{align}
|s_0| < \frac{1}{2C_n C_P(M)}
\end{align}
depends on the Poincar\'e constant $C_P(M)$ that in general is not explicit, moreover the estimates of the $S$-resolvent operator are not optimal as well  the non-optimality originates from the step in the proof of
Theorem~\ref{Th:Main_general_nonneg_POIN} where the term $(k_{\min}+|s|^{2})\|F\|_{2}^{2}\geq 0$ was discarded, so the resulting bound makes no use of the contributions of $k_{\min}$ or $|s|^{2}$.
\end{remark}

\begin{theorem}[Negative scalar curvature using Poincar\'e inequality] \label{Th:Main_general_neg_POIN}
Let $(M,g)$ be a compact spin manifold $M$ with boundary $\partial M$ and a Riemannian metric $g$ with scalar curvature $\kappa$ and let us assume that
\begin{equation}\label{kappabarNEGAT}
k_{min}:= \frac{1}{4}\min_{x\in M}\kappa(x) <0 .
\end{equation}
Let $s\in\BR^{n+1}$, and assume that
\begin{align}\label{STIME_SU_ESSE_DIR_POINDDD}
1+k_{min}C_{P}^{2}(M) -2\modulo{s_0}C_n C_P(M)>0
\end{align}
where $C_n$ is given by (\ref{ineqLemma}) and $C_P(M)$ is the Poincar\'e constant.
Then, for every $f\in L^2(M,\mathcal{E})$ there exists a unique $F_f\in H_0^1(M,\mathcal{E})$ such that
\begin{equation*}
   q_s(F_f,G) = \inner{f,G}_{2}, \qquad \text{for all } G\in H_0^1(M,\mathcal{E}).
\end{equation*}
Moreover, for $s\in \mathbb{R}^{n+1}$ that satisfies estimate we have
\begin{equation}\label{quatSresorig_DIR_EST_POIN2}
\|S_R^{-1}(s,D)f\|_2\leq C_{P}(M)\frac{C_n+C_{P}(M)|s| }{1+k_{min}C_P^2(M)-2\modulo{s_0}C_n C_P(M)}\norm{f}_2.
\end{equation}
\end{theorem}

\begin{proof}
We just consider the coercivity condition; the continuity condition is standard.
We observe that (assuming $F=G$ in the inner product of $q_s(F,G)$), and recalling that $q_{s}(F,F)\in\mathbb{R}$ since $\kappa$ is real-valued and the Dirac operator $D$ is symmetric:
\begin{align} \label{Eq:Coer_General_POINCARE2}
    q_s(F,F) &= \inner{\nabla^\tau F,\nabla^\tau F}_2 + \tfrac{1}{4}\inner{\kappa F,F}_2 - 2s_0\inner{DF ,F}_2 + |s|^2\norm{F}_2^2 \nonumber \\
    &\geq \norm{\nabla^\tau F}_2^2  + k_{\min} \norm{F}_2^2 - 2|s_0|C_n\, \norm{\nabla^\tau F}_2\norm{F}_2 + |s|^2\norm{F}_2^2.
\end{align}
We apply the Poincar\'e inequality to the term $2|s_0|C_n\, \norm{\nabla^\tau F}_2\norm{F}_2$ and use the fact that $|s|^2\norm{F}_2^2\geq 0$ in addition, since $k_{\min}<0$ by hypothesis~\eqref{kappabarNEGAT}, the term $k_{\min}\|F\|_{2}^{2}$ is itself non-positive, and we bound it from below by the Poincar\'e inequality $\|F\|_{2}^{2}\leq C_{P}^{2}(M)\|\nabla^{\tau}F\|_{2}^{2}$ with the inequality reversed due to the negative sign:
\[
   k_{\min}\,\|F\|_{2}^{2} \;\geq\; k_{\min}\,C_{P}^{2}(M)\,\|\nabla^{\tau}F\|_{2}^{2}\qquad(\text{since }k_{\min}<0).
\]
Combining these three steps,
\begin{align*}
    q_s(F,F) &\geq \norm{\nabla^\tau F}_2^2 \p{1 + k_{\min}C_P^2(M) - 2|s_0|C_n C_P(M)}.
\end{align*}
We require the coefficient of $\norm{\nabla^\tau F}_2^2$ to be strictly positive:
\[
1+k_{\min}C_P^2(M)-2|s_0|C_n C_P(M)>0.
\]
Under this condition, $q_s$ is coercive on $H_{0}^{1}(M,\mathcal{E})$ endowed with the gradient norm $\|\cdot\|_{H_{0}^{1}}=\|\nabla^{\tau}\,\cdot\,\|_{2}$, which is equivalent to the full $H^{1}$-norm by Poincar\'e:
\begin{equation}\label{POINCAREDDDDC2}
    q_s(F,F) \geq \Big(1+k_{\min}C_P^2(M) -2|s_0|C_n C_P(M) \Big)\norm{\nabla^\tau F}_2^2.
\end{equation}
Reasoning as before we get estimate (\ref{quatSresorig_DIR_EST_POIN2}).
Explicitly: testing the defining equation $q_{s}(F_{f},G)=\inner{f,G}_{2}$ with $G=F_{f}$ and using~\eqref{POINCAREDDDDC2} together with $q_{s}(F_{f},F_{f})\geq 0$ being real,
\begin{equation}\label{Eq:Chain_Inequality_Poincare}
\begin{split}
\bigl(1+k_{\min}C_{P}^{2}(M)-2|s_{0}|C_{n}C_{P}(M)\bigr)\,\|\nabla^{\tau}F_{f}\|_{2}^{2} &\leq q_{s}(F_{f},F_{f}) \\
&\leq \|f\|_{2}\,\|F_{f}\|_{2} \\
&\leq C_{P}(M)\,\|f\|_{2}\,\|\nabla^{\tau}F_{f}\|_{2},
\end{split}
\end{equation}
where the last step uses Poincar\'e. Dividing by $\|\nabla^{\tau}F_{f}\|_{2}$ (the case $F_{f}=0$ being trivial) gives
\[
   \|\nabla^{\tau}F_{f}\|_{2} \;\leq\; \frac{C_{P}(M)\,\|f\|_{2}}{1+k_{\min}C_{P}^{2}(M)-2|s_{0}|C_{n}C_{P}(M)},
\]
and Poincar\'e once more gives
\[
   \|F_{f}\|_{2} \;\leq\; C_{P}(M)\,\|\nabla^{\tau}F_{f}\|_{2} \;\leq\; \frac{C_{P}^{2}(M)\,\|f\|_{2}}{1+k_{\min}C_{P}^{2}(M)-2|s_{0}|C_{n}C_{P}(M)}.
\]
The $S$-resolvent estimate then follows from:
\begin{align*}
   \|S_{R}^{-1}(s,D)f\|_{2} &\leq \|DF_{f}\|_{2}+|s|\,\|F_{f}\|_{2} \;\leq\; C_{n}\,\|\nabla^{\tau}F_{f}\|_{2}+|s|\,\|F_{f}\|_{2} \\
   &\leq\frac{C_{P}(M)(C_{n}+C_{P}(M)|s|)}{1+k_{\min}C_{P}^{2}(M)-2|s_{0}|C_{n}C_{P}(M)}\,\|f\|_{2},
\end{align*}
which is precisely~\eqref{quatSresorig_DIR_EST_POIN2}.
\end{proof}

\begin{remark}
It should be noted that points $s$ that satisfy the condition~\eqref{STIME_SU_ESSE_DIR_POINDDD} can form an empty set. Indeed, it will not be empty if
\[
1+k_{\min}C_P^2(M)>0 \iff \min_{x\in M} \kappa(x) > -4/C_P^2(M).
\]

If this condition holds, then we have that the resolvent set contains the vertical strip of points $s$ such that
\[
|s_0| < \frac{1+k_{\min}C_P^2(M)}{2C_n C_P(M)},
\]
which depends only on the real part $s_{0}$ and is independent of $s_{1}$, giving rise to a spectral gap.
\end{remark}


\section{The $S$-resolvent estimates with Robin-like boundary conditions}\label{ROBINSECTION}

We will study the elliptic boundary value problem with homogeneous Robin-like boundary condition
\begin{equation} \label{EQ:BVP_ROBIN}
		\begin{cases}
			Q_s(D)F = f, \quad \text{in } M, \\
		\nabla_N^\tau F\vert_{\partial M}+bF\vert_{\partial M} = 0,
		\end{cases}
\end{equation}
where $Q_s(D) = D^2 - 2s_0 D + \modulo{s}^2$, while $N$ is the unit normal vector field to the boundary and $b$ is a given bounded real valued function. Our aim is to study the invertibility of the operator $Q_s(D)$ in (\ref{EQ:BVP_ROBIN}) in the weak sense.  That is, we ask whether the problem (\ref{EQ:BVP_ROBIN}) admits a unique weak solution for every $f\in L^2(M,\mathcal{E})$.

\medskip
In order to derive the weak formulation, we have to integrate by parts the second order part of $Q_s(D)$, which gives us a Green's-like formula for $-\Delta_\tau$. For $F,G\in H^1(M,\mathcal{E})$ and for every $s\in\BR^{n+1}$, we consider the following bilinear form associated with the boundary value problem (\ref{EQ:BVP_ROBIN})
\begin{align} \label{EQ:Form_general_ROBIN}
   q_s^R(F,G) & = \int_M \Sc{\overline{\nabla^\tau F} \nabla^\tau G} \: \dif\text{vol}
- \int_{\partial M} \Sc{\overline{\nabla_N^\tau F} G} \: \widetilde{\dvol}
+ \frac{1}{4}\int_{M} \kappa(x)\Sc{\overline{F}G} \: \dif\text{vol} \nonumber \\
& \hspace*{4cm} - 2s_0\int_{M} \Sc{\overline{DF}G} \: \dif\text{vol} + \modulo{s}^2\int_{M} \Sc{\overline{F}G} \: \dif\text{vol}
\end{align}
obtained by applying Theorem \ref{Th:Green_like_formula} to the second order part of $Q_s(D)$ in \eqref{EQ:BVP_ROBIN}, i.e.\ to $\Delta_\tau$.
Keeping in mind Problem \ref{PROB_DIR_FULL_NORM_DIR} the form $q_s^R(F,G)$ in (\ref{EQ:Form_general_ROBIN}) can be written as
\begin{equation}\label{qROBINONE}
q_s^R(F,G)=q_s(F,G)- \int_{\partial M} \Sc{\overline{\nabla_N^\tau F} G} \: \widetilde{\dvol},
\end{equation}
where $q_s(F,G)$ is given by (\ref{EQ:Form_gen_inner}). Now we take into account the Robin-like boundary conditions and for simplicity we set
$$
\inner{b{\rm Tr} (F),{\rm Tr}(G)}_{2, \partial M}:=\int_{\partial M} \Sc{\overline{b F} G} \: \widetilde{\dvol}
$$
where ${\rm Tr}\,F$ stands for the boundary trace of the function $F$ and
$$
-\int_{\partial M} \Sc{\overline{b F} G} \: \widetilde{\dvol}=\int_{\partial M} \Sc{\overline{\nabla_N^\tau F} G},
$$
so we get
\begin{equation}\label{qROBIN}
q_s^R(F,G)=q_s(F,G)-\inner{{\rm Tr} (bF),{\rm Tr}(G)}_{2, \partial M}.
\end{equation}

\begin{problem}\label{PROB_DIR_FULL_NORM_ROBIN}
Let $D$ be the Dirac operator in $\eqref{Eq:General_Dirac}$ and $\kappa$ be the scalar curvature. Let $q_s^R(F,G)$ be the form defined in (\ref{qROBIN}) and assume that
$$
\dom (q_s^R) := H^1(M,\mathcal{E})\times H^1(M,\mathcal{E}).
$$
Show that for some values of the spectral parameter $s\in\BR^{n+1}$, for every $f\in L^2(M,\mathcal{E})$ there exists a unique solution $F_f\in H^1(M,\mathcal{E})$ such that
\begin{equation}\label{WEAK_EQUATION_DIRIC}
   q_s^R(F_f,G) = \inner{f,G}_2, \quad \text{for all } G\in H^1(M,\mathcal{E}).
\end{equation}
Furthermore, determine $L^2$- and $\nabla^\tau$-estimates of $F_f$, depending on the parameter $s\in\BR^{n+1}$.
\end{problem}

\begin{remark}
In the following result we consider the positive part of $b$, that is $\max\{0,b\}$, because the destabilizing contribution of the Robin term $-\inner{b\,\mathrm{Tr}F,\mathrm{Tr}F}_{2,\partial M}$ to the coercivity arises from the positive part of $b$, since:
\begin{equation}
\int_{\partial M}b|\mathrm{Tr}F|^{2} \leq \int_{\partial M}b_{+}|\mathrm{Tr}F|^{2} \leq \|b_{+}\|_{L^{\infty}}\|\mathrm{Tr}F\|_{2}^{2}.
\end{equation}
\end{remark}

\begin{theorem}\label{Th:Main_general_ROBIN}
Let $(M,g)$ be a compact spin manifold $M$ with boundary $\partial M$ and a Riemannian metric $g$ with scalar curvature $\kappa(x)$ and define:
\begin{equation}\label{kappabar_robin}
k_{\min}:= \frac{1}{4}\min_{x\in M}\kappa(x), \qquad \text{and} \qquad b_{+}(x) = \max\{0,b(x)\}\ \ \text{for}\ \ b \in L^\infty(\partial M).
\end{equation}
Assume that, for $s\in\BR^{n+1}$, we have:
\begin{equation}\label{STIME_SU_ESSE_ROBIN_B_POSIT}
1- \sigma\|b_{+}\|_{L^\infty(\partial M)}>0 , \qquad \text{and} \qquad
(|s|^2 +k_{\min}- \sigma\|b_{+}\|_{L^\infty(\partial M)}) - \frac{|s_0|^2 C_n^2}{1- \sigma\|b_{+}\|_{L^\infty(\partial M)}} >0,
\end{equation}
where $\sigma$ is the constant in (\ref{ESTIM_bound_TERMTAU}). Setting, for brevity,
\begin{equation}\label{alefbet}
   \alpha\;:=\; 1-\sigma\|b_{+}\|_{L^{\infty}(\partial M)}, \qquad \beta(s) \;:=\; |s|^{2}+k_{\min}-\sigma\|b_{+}\|_{L^{\infty}(\partial M)},
\end{equation}
then, for every $f\in L^2(M,\mathcal{E})$ there exists a unique $F_f\in H^1(M,\mathcal{E})$ that satisfies equation (\ref{WEAK_EQUATION_DIRIC}) and we have:
\begin{equation}\label{quatSresorig_ROBIN_A_POSIT}
   \|S_R^{-1}(s,D)f\|_2 \;\leq\;
   \frac{C_{n}\,\sqrt{\beta(s)}\;+\;|s|\,\sqrt{\alpha}}
        {\beta(s)\,\sqrt{\alpha}\;-\;C_{n}|s_{0}|\,\sqrt{\beta(s)}}
   \;\|f\|_2 .
\end{equation}
\end{theorem}

\begin{proof}
We will verify that the bilinear form $q_s^R$ satisfies the assumptions of Lax--Milgram lemma.
We have that
\[
\inner{{\rm Tr} (bF),{\rm Tr}(G)}_{2, \partial M}:=\int_{\partial M} \Sc{\overline{b F} G} \: \widetilde{\dvol}
\]
satisfies
\begin{align}\label{ESTIM_bound_TERMTAU}
\inner{{\rm Tr}(bF),{\rm Tr}(G)}_{2, \partial M} &\leq \Vert b\Vert_{L^\infty(\partial M)}\Vert {\rm Tr}F \Vert_{L^2(\partial M)}\Vert{\rm Tr}G\Vert_{L^2(\partial M)} \\
& \leq \sigma\Vert b\Vert_{L^\infty(\partial M)}\Vert F\Vert_{H^1(M)}\Vert G\Vert_{H^1(M)}.\nonumber
\end{align}
The first inequality follows from H\"older's inequality applied to the bounded multiplier $b\in L^{\infty}(\partial M)$ and the $L^{2}$ pairing of the boundary traces; the second inequality uses the continuity of the trace operator $\mathrm{Tr}\colon H^{1}(M)\to L^{2}(\partial M)$, whose operator norm is denoted $\sigma$, applied twice.

\medskip
The continuity of the bilinear form $q_s^R$ follows from the continuity of $q_s(F,G)$ as in (\ref{EQ:Form_gen_inner}) and from estimate (\ref{ESTIM_bound_TERMTAU}), so we have:
\begin{align*}
|q_s^R(F,G)| \leq C(\modulo{s},\modulo{s_0}, k_{\max}, \sigma, \Vert b\Vert_{L^\infty(\partial M)})\norm{F}_{H^1}\norm{G}_{H^1}, \quad \text{for all } F,G\in H^1(M,\mathcal{E}),
\end{align*}
where $k_{\max}:=\frac{1}{4}\max_{x\in M}\kappa(x)$.
Let us consider the coercivity of $q_s^R$. Setting $F=G$ in the form $q_s^R(F,G)$ and using Lemma \ref{LemmaDF64}, we have:
\begin{align}\label{Eq3:Coer_General_nonneg}
   q_s^R(F,F) & = \inner{\nabla^\tau F,\nabla^\tau F}_2 + \frac{1}{4} \inner{\kappa F,F}_2 - 2s_0\inner{DF ,F}_2 + |s|^2\inner{F ,F}_2 -\inner{b{\rm Tr} (F),{\rm Tr}(F)}_{2, \partial M}\nonumber \\
   & \geq \norm{\nabla^\tau F}_2^2 + k_{\min}\norm{F}_2^2 - 2|s_0|C_n\, \norm{\nabla^\tau F}_2\norm{F}_2 + |s|^2\norm{F}_2^2-\inner{b{\rm Tr} (F),{\rm Tr}(F)}_{2, \partial M}.
\end{align}
Since $b = b_+ + b_{-}$ (with $b_{+}=\max\{0,b\}\geq 0$ the positive part and $b_{-}=b-b_{+}\leq 0$ the negative part), we have that the last term satisfies:
\[
\inner{b{\rm Tr} (F),{\rm Tr}(F)}_{2, \partial M} \leq \inner{b_+{\rm Tr} (F),{\rm Tr}(F)}_{2, \partial M},
\]
since $b\leq b_{+}$ pointwise (as $b_{-}\leq 0$), and $|\mathrm{Tr}F|^{2}\geq 0$. This inequality is relevant because we need an upper bound for $\inner{b\mathrm{Tr}F,\mathrm{Tr}F}_{2,\partial M}$ in order to obtain a lower bound for the term $-\inner{b\mathrm{Tr}F,\mathrm{Tr}F}_{2,\partial M}$ in $q_{s}^{R}(F,F)$. Indeed,
\begin{equation}\label{Eq:Robin_Boundary_Inequality}
\begin{split}
-\inner{b\,\mathrm{Tr}F,\mathrm{Tr}F}_{2,\partial M} &\geq -\inner{b_{+}\mathrm{Tr}F,\mathrm{Tr}F}_{2,\partial M} \\
&\geq -\|b_{+}\|_{L^{\infty}(\partial M)}\,\|\mathrm{Tr}F\|_{L^{2}(\partial M)}^{2} \\
&\geq -\sigma\|b_{+}\|_{L^{\infty}(\partial M)}\,\|F\|_{H^{1}(M)}^{2},
\end{split}
\end{equation}
where the second step is H\"older with $b_{+}\geq 0$ and the third uses the trace inequality.
Then we get using Young's inequality and the trace inequality $\|\mathrm{Tr}F\|_{L^{2}(\partial M)}^{2}\leq\sigma\|F\|_{H^{1}(M)}^{2}$:
\begin{align}\label{Eq2:lower_bound_general_nonneg}
q_s^R(F,F) &\geq \|\nabla^\tau F\|_2^2 + k_{\min}\|F\|_2^2 - 2C_n \left(\frac{1}{2\delta}\|\nabla^\tau F\|_2^2 + \frac{\delta}{2}|s_0|^2\|F\|_2^2 \right) + |s|^2\|F\|_2^2 \nonumber \\
&\hspace{1.5cm}- \sigma\|b_+\|_{L^\infty(\partial M)} \|F\|_{H^1(M)}^2 \nonumber \\
& \geq \left(1-\frac{C_n}{\delta}\right) \|\nabla^\tau F\|_2^2 + \left(|s|^2 + k_{\min} - |s_0|^2 \delta C_n \right) \|F\|_2^2 \nonumber \\
& \quad - \sigma\|b_+\|_{L^\infty(\partial M)} \|F\|^2_2- \sigma\|b_+\|_{L^\infty(\partial M)}\|\nabla^\tau F\|^2_2.
\end{align}
Here Young's inequality is applied in the form $ab\leq a^{2}/(2\delta)+\delta b^{2}/2$, with $a=\|\nabla^{\tau}F\|_{2}$ and $b=|s_{0}|\,\|F\|_{2}$, which yields $$2|s_{0}|C_{n}\|\nabla^{\tau}F\|_{2}\|F\|_{2}\leq (C_{n}/\delta)\|\nabla^{\tau}F\|_{2}^{2}+\delta|s_{0}|^{2}C_{n}\|F\|_{2}^{2}.$$
On the boundary side, we split
\[
   \sigma\|b_{+}\|_{L^{\infty}(\partial M)}\,\|F\|_{H^{1}(M)}^{2}\;=\;\sigma\|b_{+}\|_{L^{\infty}(\partial M)}\,\|F\|_{2}^{2}\;+\;\sigma\|b_{+}\|_{L^{\infty}(\partial M)}\,\|\nabla^{\tau}F\|_{2}^{2},
\]
so that each summand can be absorbed into the corresponding $\|F\|_{2}^{2}$ or $\|\nabla^{\tau}F\|_{2}^{2}$ coefficient.
We then deduce
\begin{align}\label{Eq:lower_bound_general_nonneg11}
q_s^R(F,F) & \geq \left(1-\frac{C_n}{\delta}- \sigma\|b_+\|_{L^\infty(\partial M)}\right) \|\nabla^\tau F\|_2^2 + \left(|s|^2 + k_{\min}- \sigma\|b_+\|_{L^\infty(\partial M)} - |s_0|^2 \delta C_n \right) \|F\|_2^2.
\end{align}
The two coefficients into parenthesis in (\ref{Eq:lower_bound_general_nonneg11}) must be positive and so
\begin{equation}\label{xAandB_DELTA}
A(\delta):=1-\frac{C_n}{\delta}- \sigma\|b_+\|_{L^\infty(\partial M)}>0, \qquad
B(\delta):=|s|^2+k_{\min}- \sigma\|b_+\|_{L^\infty(\partial M)}  - |s_0|^2\delta C_n >0.
\end{equation}
Since $\delta$ and $C_n$ are positive constants, $A(\delta)$ cannot be positive if $1 - \sigma \|b_+\|_{L^\infty(\partial M)}$ is negative. Therefore, the following condition must hold:
\[
1 - \sigma \|b_+\|_{L^\infty(\partial M)}> 0.
\]
Observe that in compact notation
$$
A(\delta)=\alpha-C_{n}/\delta\ \ \ \text{and} \ \ \  B(\delta)=\beta(s)-|s_{0}|^{2}C_{n}\delta,
$$
where $\alpha$ and $\beta$ are defined in (\ref{alefbet}).
 Both are continuous, monotone functions of $\delta$ on $(0,+\infty)$: $A$ is strictly increasing (with $A'(\delta)=C_{n}/\delta^{2}>0$), while $B$ is strictly decreasing (with $B'(\delta)=-|s_{0}|^{2}C_{n}<0$ when $s_{0}\neq 0$). Moreover $A(\delta)\to-\infty$ as $\delta\to 0^{+}$ and $A(\delta)\to\alpha$ as $\delta\to+\infty$, while $B(\delta)\to\beta(s)$ as $\delta\to 0^{+}$ and $B(\delta)\to-\infty$ as $\delta\to+\infty$.
Similarly to the proof of Theorem~\ref{Th:Main_general}, by isolating $\delta$, when $s_0 \neq 0$ we find that
\begin{equation}\label{eq:range_delta_robin}
   \frac{C_n}{(1 - \sigma \|b_+\|_{L^\infty(\partial M)})} <\delta < \frac{|s|^2+k_{\min}- \sigma\|b_+\|_{L^\infty(\partial M)}}{C_n|s_0|^2}.
\end{equation}
In order for such a $\delta>0$ to exist, the upper bound must be bigger than the lower bound, which leads to
\[
(|s|^2+k_{\min}- \sigma\|b_+\|_{L^\infty(\partial M)})>\frac{C_n^2 s_0^2}{(1 - \sigma \|b_+\|_{L^\infty(\partial M)})}.
\]
The same inequality must hold when $s_0 = 0$.
Hence, we have proven that $q_s^{R}$ is coercive in $H^1(M,\mathcal{E})$ when conditions (\ref{STIME_SU_ESSE_ROBIN_B_POSIT}) are satisfied.
Fixing now any function $f\in L^2(M,\mathcal{E})$, we can consider the corresponding functional
\begin{equation*}
   \varphi_f(G) := \inner{f,G}_2, \quad G\in H^1(M,\mathcal{E})
\end{equation*}
as in Problem \ref{PROB_DIR_FULL_NORM_DIR}.
Then by Remark \ref{LINFUNCCONST}, this functional is bounded in $H^1(M,\mathcal{E})$, i.e.,
\begin{equation*}
   \modulo{\varphi_f(G)} = \modulo{\inner{f,G}_{L^2}} \leq \norm{f}_2\norm{G}_2 \leq \norm{f}_2\norm{G}_{H^1}, \quad G\in H^1(M,\mathcal{E}).
\end{equation*}
Hence, the assumptions of Lax--Milgram lemma are satisfied and so there exists a unique weak solution $F_f\in H^1(M,\mathcal{E})$ such that
\begin{equation} \label{Eqq:test_funct_general_nonneg}
   q_s^{R}(F_f,G) = \varphi_f(G) = \inner{f,G}_2, \quad \text{for all } G\in H^1(M,\mathcal{E}).
\end{equation}
For the $L^2$-estimate in (\ref{quatSresorig_ROBIN_A_POSIT}), we test (\ref{Eqq:test_funct_general_nonneg}) with $G = F_f$.
Still considering Problem \ref{PROB_DIR_FULL_NORM_ROBIN} we have the estimate
\begin{equation}\label{Eq:norm_sc_bound_general_nonneg_R}
   |q_s^{R}(F_f,F_f)| = |\inner{f,F_f}_2| \leq \norm{f}_2\norm{F_f}_2.
\end{equation}
Combining this inequality with the coercivity estimate we get
\begin{equation}\label{ESTIM_PSEUDO_RESOL_HHH}
   A(\delta)\norm{\nabla^\tau F_{f}}_2^2 +B(\delta)\norm{F_{f}}_2^2 \leq q_s^{R}(F_f,F_f) \leq \norm{f}_2\norm{F_f}_2.
\end{equation}
The leftmost inequality combines the coercivity estimate~\eqref{Eq:lower_bound_general_nonneg11} with the fact that $q_{s}^{R}(F_{f},F_{f})$ is real (since each of its summands is real, by symmetry of $D$ and the realness of $\kappa$ and of $b$) and, by the coercivity itself, non-negative on the range~\eqref{eq:range_delta_robin}, so that $|q_{s}^{R}(F_{f},F_{f})|=q_{s}^{R}(F_{f},F_{f})$.
We finally discuss the optimisation in $\delta$.
As in the proof of Theorem~\ref{Th:Main_general}, from~\eqref{ESTIM_PSEUDO_RESOL_HHH} we obtain the two pointwise estimates (in $\delta$ ranging in the interval~\eqref{eq:range_delta_robin}):
\[
   \norm{F_{f}}_{2} \;\leq\; \frac{\norm{f}_{2}}{B(\delta)}, \qquad \norm{\nabla^{\tau}F_{f}}_{2} \;\leq\; \frac{\norm{f}_{2}}{\sqrt{A(\delta)B(\delta)}}.
\]
The first follows by dropping the non-negative summand $A(\delta)\|\nabla^{\tau}F_{f}\|_{2}^{2}$ in~\eqref{ESTIM_PSEUDO_RESOL_HHH} and dividing by $\|F_{f}\|_{2}$; the second by dropping the non-negative summand $B(\delta)\|F_{f}\|_{2}^{2}$, then bounding $\|F_{f}\|_{2}$ by means of the first estimate, and finally taking the square root.
Recalling that $$S_{R}^{-1}(s,D)f = -(D-\overline{s}\mathcal{I})F_{f}$$ and using Lemma~\ref{LemmaDF64} we have
\[
   \|S_{R}^{-1}(s,D)f\|_{2} \;\leq\; C_{n}\,\|\nabla^{\tau}F_{f}\|_{2} + |s|\,\|F_{f}\|_{2} \;\leq\; \Bigl(\frac{C_{n}}{\sqrt{A(\delta)B(\delta)}} + \frac{|s|}{B(\delta)}\Bigr)\,\|f\|_{2}.
\]
We now have to choose the value of $\delta$ optimally. Since the bound holds for every admissible $\delta$, we minimise
\[
   G(\delta):=\frac{C_{n}}{\sqrt{A(\delta)B(\delta)}}+\frac{|s|}{B(\delta)}
\]
over the open interval~\eqref{eq:range_delta_robin}. Setting $G'(\delta)=0$ leads, after clearing denominators and squaring, to a quintic polynomial equation in $\delta$, exactly as observed in the proof of Theorem~\ref{Th:Main_general}, its roots are not available in closed form. We therefore proceed as in that proof and evaluate $G$ at the point $\delta^{*}$ that maximises $A(\delta)B(\delta)$, which by elementary calculus is
\[
   \delta^{*} \;=\; \sqrt{\frac{|s|^{2}+k_{\min}-\sigma\|b_{+}\|_{L^{\infty}(\partial M)}}{|s_{0}|^{2}\bigl(1-\sigma\|b_{+}\|_{L^{\infty}(\partial M)}\bigr)}} \;=\; \frac{\sqrt{\beta(s)}}{|s_{0}|\sqrt{\alpha}},
\]
with $\alpha$ and $\beta(s)$ as in~\eqref{alefbet}.
To verify now that, with $\gamma:=|s_{0}|^{2}C_{n}$ so that $B(\delta)=\beta-\gamma\delta$, we have
\[
   \frac{d}{d\delta}\bigl[A(\delta)B(\delta)\bigr]\;=\;A'(\delta)B(\delta)+A(\delta)B'(\delta)\;=\;\frac{C_{n}\beta}{\delta^{2}}\;-\;\alpha\gamma\;=\;\frac{C_{n}\beta(s)}{\delta^{2}}\;-\;\alpha\,|s_{0}|^{2}\,C_{n},
\]
which vanishes precisely at
$$
\delta=\delta^{*}=\sqrt{\beta(s)/(\alpha|s_{0}|^{2})}
$$
 the second derivative
 $$\frac{d^{2}}{d\delta^{2}}[AB]=-2C_{n}\beta(s)/\delta^{3}<0$$
 confirms a maximum.
 Both endpoints of~\eqref{eq:range_delta_robin} make $A(\delta)B(\delta)$ vanish
 (since $A$ vanishes at the left endpoint and $B$ at the right), so $\delta^{*}$ is the global maximiser. Furthermore $\delta^{*}$ lies inside the open interval~\eqref{eq:range_delta_robin}: the standing hypothesis $\alpha\beta(s)>C_{n}^{2}|s_{0}|^{2}$, that is $\sqrt{\alpha\beta}>C_{n}|s_{0}|$, gives both
\[
   \delta^{*}\;=\;\frac{\sqrt{\beta}}{|s_{0}|\sqrt{\alpha}}\;>\;\frac{C_{n}}{\alpha}\;=\;\frac{C_{n}}{1-\sigma\|b_{+}\|_{L^{\infty}(\partial M)}}\quad\text{(left endpoint)},
\]
and
\[
   \delta^{*}\;<\;\frac{\beta}{|s_{0}|^{2}C_{n}}\;=\;\frac{|s|^{2}+k_{\min}-\sigma\|b_{+}\|_{L^{\infty}(\partial M)}}{C_{n}|s_{0}|^{2}}\quad\text{(right endpoint)}.
\]
A direct computation gives
\[
   A(\delta^{*})\,B(\delta^{*}) \;=\; \bigl(\sqrt{\alpha\,\beta(s)} - C_{n}|s_{0}|\bigr)^{2}.
\]
Indeed, plugging $\delta^{*}=\sqrt{\beta}/(|s_{0}|\sqrt{\alpha})$ into $A$ and $B$:
\begin{align*}
   A(\delta^{*}) &\;=\; \alpha-\frac{C_{n}}{\delta^{*}}\;=\;\alpha-\frac{C_{n}\,|s_{0}|\sqrt{\alpha}}{\sqrt{\beta}}\;=\;\frac{\sqrt{\alpha}\,\bigl(\sqrt{\alpha\beta}-C_{n}|s_{0}|\bigr)}{\sqrt{\beta}},\\
   B(\delta^{*}) &\;=\; \beta-|s_{0}|^{2}C_{n}\,\delta^{*}\;=\;\beta-\frac{|s_{0}|C_{n}\sqrt{\beta}}{\sqrt{\alpha}}\;=\;\frac{\sqrt{\beta}\,\bigl(\sqrt{\alpha\beta}-C_{n}|s_{0}|\bigr)}{\sqrt{\alpha}},
\end{align*}
and the product simplifies to $(\sqrt{\alpha\beta}-C_{n}|s_{0}|)^{2}$. Under the standing hypothesis $\sqrt{\alpha\beta(s)} > C_{n}|s_{0}|$ (which is precisely the second condition in~\eqref{STIME_SU_ESSE_ROBIN_B_POSIT}, $\alpha\beta>C_{n}^{2}|s_{0}|^{2}$), taking positive square roots,
\[
   \sqrt{A(\delta^{*})\,B(\delta^{*})} \;=\; \sqrt{\alpha\,\beta(s)} - C_{n}|s_{0}|
\]
and similarly we get
\[
   B(\delta^{*}) \;=\; \sqrt{\frac{\beta(s)}{\alpha}}\,\bigl(\sqrt{\alpha\,\beta(s)} - C_{n}|s_{0}|\bigr).
\]
Substituting these values,
\begin{equation}\label{eq:proof_alphabeta_form}
   \|S_{R}^{-1}(s,D)f\|_{2} \;\leq\; G(\delta^{*})\,\|f\|_{2} \;=\; \frac{1}{\sqrt{\alpha\beta}-C_{n}|s_{0}|}\Bigl(C_{n}+|s|\sqrt{\frac{\alpha}{\beta}}\Bigr)\,\|f\|_{2}.
\end{equation}
Finally, multiplying numerator and denominator of the right-hand side of~\eqref{eq:proof_alphabeta_form} by $\sqrt{\beta(s)}$ and using $\sqrt{\alpha\beta}\,\sqrt{\beta}=\beta\sqrt{\alpha}$, we obtain the single fraction
\[
   \|S_{R}^{-1}(s,D)f\|_{2} \;\leq\;
   \frac{C_{n}\,\sqrt{\beta(s)}+|s|\,\sqrt{\alpha}}{\beta(s)\,\sqrt{\alpha}-C_{n}|s_{0}|\,\sqrt{\beta(s)}}\,\|f\|_{2},
\]
which is precisely \eqref{quatSresorig_ROBIN_A_POSIT}.
\end{proof}

\begin{remark}[Recovering estimates for homogeneous Neumann problem]
Substituting the explicit expressions \eqref{alefbet} for $\alpha$ and $\beta(s)$ into the estimate \eqref{quatSresorig_ROBIN_A_POSIT}, and writing for brevity
\[
   \beta_{\sigma} \;:=\; \sigma\|b_{+}\|_{L^{\infty}(\partial M)},
\]
so that $\alpha= 1-\beta_{\sigma}$ and $\beta(s) = |s|^{2}+k_{\min}-\beta_{\sigma}$, the bound \eqref{quatSresorig_ROBIN_A_POSIT} reads explicitly
\begin{equation}\label{quatSresorig_ROBIN_subst_single}
   \|S_R^{-1}(s,D)f\|_2 \;\leq\;
   \frac{C_{n}\,\sqrt{|s|^{2}+k_{\min}-\beta_{\sigma}}\;+\;|s|\,\sqrt{1-\beta_{\sigma}}}
        {\sqrt{1-\beta_{\sigma}}\,\bigl(|s|^{2}+k_{\min}-\beta_{\sigma}\bigr)\;-\;C_{n}|s_{0}|\,\sqrt{|s|^{2}+k_{\min}-\beta_{\sigma}}}
   \;\|f\|_2 .
\end{equation}
In particular, in the homogeneous Neumann problem $\nabla_N^\tau F\vert_{\partial M}=0$, for the case $b\equiv 0$ (whence $\beta_{\sigma}=0$, $\alpha=1$ and $\beta(s)=|s|^{2}+k_{\min}$), in \eqref{quatSresorig_ROBIN_subst_single} reduces to
\[
   \|S_R^{-1}(s,D)f\|_2 \;\leq\;
   \frac{C_{n}\,\sqrt{|s|^{2}+k_{\min}}+|s|}{\,|s|^{2}+k_{\min}-C_{n}|s_{0}|\,\sqrt{|s|^{2}+k_{\min}}\,}\;\|f\|_2,
\]
which can equivalently be written as
\[
   \|S_{R}^{-1}(s,D)f\|_{2} \leq \frac{1}{\sqrt{|s|^{2}+k_{\min}}-C_{n}|s_{0}|}\Bigl(C_{n}+\frac{|s|}{\sqrt{|s|^{2}
   +k_{\min}}}\Bigr)\|f\|_{2}.
\]
This bound in $H^1(M,\mathcal{E})$ is precisely the estimate  \eqref{quatSresorig_ROBIN_A_POSIT} of Theorem~\ref{Th:Main_general}.
\end{remark}

\begin{remark}
The existence of the $S$-resolvent operator under Robin boundary conditions depends on the parameter
\[
\Lambda := k_{\min} - \sigma\|b_+\|_{L^\infty(\partial M)},
\]
that we call the effective curvature bound. If in addition
the estimate holds
$$1 - \sigma\|b_+\|_{L^\infty(\partial M)}>0$$
the $S$-resolvent set contains all points $s = s_0 + s_1 J$ such that
\[
(|s|^2 +k_{\min}- \sigma\|b_{+}\|_{L^\infty(\partial M)}) - \frac{|s_0|^2 C_n^2}{1- \sigma\|b_{+}\|_{L^\infty(\partial M)}} >0.
\]
Similarly to the Dirichlet boundary conditions, we have three qualitatively different behaviours depending on the sign of the effective curvature bound $\Lambda$. The qualitative picture of the resolvent set is analogous to Figure~\ref{fig:resolvent_sets} taking $\Lambda$ into account.
\end{remark}

\section{Concluding remarks}

The estimates established in this work show the bi-sectoriality of the spinor Dirac operator within the spectral theory on the $S$-spectrum.
These estimates are crucial in the sense that they are the bounds required by the $S$-resolvent operators to define the $H^{\infty}$-functional calculus in this setting.
We recall that the $H^{\infty}$-functional calculus for complex operators on Banach spaces, introduced by A.~McIntosh \cite{McI1} and developed in \cite{MC06, MC10, MC97, MC98}, is based on resolvent estimates of similar type. This calculus is important across many fields, and for comprehensive treatments we refer to the monographs \cite{Haase, HYTONBOOK1, HYTONBOOK2}.

\medskip
The spectral theory on the $S$-spectrum is based on slice hyperholomorphic functions and has its own $H^{\infty}$-functional calculus \cite{ACS2016, AlpayColSab2020, CGK, CGdiffusion2018, ColSab2006, ColomboSabadiniStruppa2011, FJBOOK, JONADIRECT, JONAMEM}. Parallel to it, in the Clifford setting, there is also a spectral theory based on the monogenic spectrum \cite{JM}, inspired by monogenic functions \cite{DSS}, that also possesses a corresponding $H^{\infty}$-functional calculus, see the monographs \cite{JBOOK, TAOBOOK}.

\medskip
{\em From the gradient with non-constant coefficients to the Dirac operator.}
A central motivation for proving the bi-sectoriality estimates is to extend, to the spinor Dirac operator, investigations previously carried out for the gradient operator with non-constant coefficients.
Precisely, let $e_1,\dots,e_n$ be the imaginary units of the Clifford algebra $\mathbb{R}_n$ and assume that the function $a_1,\dots,a_n:\Omega\to\mathbb{R}$ satisfy suitable regularity and growth conditions on a smooth domain $\Omega\subseteq\mathbb{R}^n$.
In \cite{Gradient}  was proved the bi-sectoriality estimates for the operator
\begin{equation}\label{Eq_T}
   \nabla_a = \sum_{i=1}^{n} e_i\, a_i(x)\, \frac{\partial}{\partial x_i}, \qquad x\in\Omega,
\end{equation}
 under various boundary conditions and thanks to the theory
developed in \cite{Quadratic, mantovani, bisectorial}) it was possible to define the $H^\infty$-functional calculus that include the fractional powers of operators.  The gradient $\nabla_a$ with non-constant coefficients represents the Fourier (or Fick) law for diffusion in non-homogeneous materials, and the spectral theory on the $S$-spectrum allows one to define the corresponding fractional diffusion law through the fractional powers of $\nabla_a$.
For higher-order operators some results are contained in the paper \cite{BARACCO}.

\medskip
We point out a fundamental difference between complex and Clifford setting regarding fractional powers. Also in this noncommutative setting, fractional powers generate non-local operators, but their definition requires overcoming obstructions that do not appear in the classical complex spectral theory and that are due to the non-commutative setting. The obstruction consists on the fact that the function $s\mapsto s^{\alpha}$ is not slice hyperholomorphic on the negative real axis, and in the hypercomplex setting of $\mathbb{R}_n$ this axis of non-holomorphicity cannot be rotated away.
 For the fractional powers of Clifford operators, the obstruction has been overcome in \cite{fracPQ} by means of the pair of functions
\begin{equation}\label{Eq_palpha_qalpha}
   p_\alpha:=\begin{cases} s^\alpha, & {\rm Sc}(s)>0, \\ -(-s)^\alpha, & {\rm Sc}(s)<0, \end{cases}\qquad
   q_\alpha:=\begin{cases} s^\alpha, & {\rm Sc}(s)>0, \\ (-s)^\alpha, & {\rm Sc}(s)<0, \end{cases}
\end{equation}
where $s^\alpha:=e^{\alpha\ln(s)}$ for ${\rm Sc}(s)>0$, with $\ln(s):=\ln|s|+J\arg(s)$, $J\in\mathbb{S}$ such that $s\in\mathbb{C}_J$ and $\arg(s)\in(-\frac{\pi}{2},\frac{\pi}{2})$. Since $q_\alpha=(s^2)^{\frac{\alpha}{2}}$, it merely rewrites the classical fractional power of $s^2$, whereas $p_\alpha$ is the more natural notion, scaling the angle to the positive (resp.\ negative) real axis by $\alpha$ when ${\rm Sc}(s)>0$ (resp.\ ${\rm Sc}(s)<0$). The fractional powers $p_\alpha(\mathcal{T})$ and $q_\alpha(\mathcal{T})$ of a bisectorial Clifford operator $\mathcal{T}$ are then defined precisely through the $H^{\infty}$-functional calculus.
We remark that the definition of the fractional powers $p_\alpha(\mathcal{T})$ is not artificial  because for $\alpha=1$ if we apply the $H^\infty$-functional calculus to the classical gradient operator $\nabla$ we get $p_1(\nabla)=\nabla$, taking also care of the domain of the gradient operator since it is unbounded. This definition is clearly a valid extension of the notion of fractional powers for unbounded Clifford operators.

\medskip
The bi-sectoriality estimates obtained in the present paper are crucial in exactly the same sense.
Using the Clifford $H^{\infty}$-functional calculus for the Dirac operators
we can define the fractional powers $p_\alpha$ and $q_\alpha$
 for the spinor Dirac operator. In this way, the entire construction developed in \cite{Gradient} for $\nabla_a$ bi-sectoriality, $H^{\infty}$-functional calculus, and fractional powers generalizes from the gradient to the Dirac operator.

\medskip
{\em Further applications of the $S$-spectrum theory.}
We close by recalling other directions in which this spectral theory has proved central.

\smallskip
\noindent\emph{Quaternionic quantum mechanics.} The $S$-spectrum originated in the search for an appropriate notion of spectrum for the quaternionic operators of quantum mechanics, in the formulation of G.~Birkhoff and J.~von Neumann \cite{BF};\ see \cite{adler, JONAQSTUD}. Its identification was achieved using techniques from hypercomplex analysis, see the introduction in \cite{CGK} for the explanations.

\smallskip
\noindent\emph{Vector analysis.} Operators of vector analysis such as the gradient $\nabla_a$ in \eqref{Eq_T} fit naturally within this theory, as discussed above.

\smallskip
\noindent\emph{Fine structures on the $S$-spectrum.} A more recent branch rests on the factorization of the second operator in the Fueter-Sce extension theorem \cite{ColSabStrupSce, Fueter, Sce, TaoQian1}. These \emph{fine structures} concern classes of functions defined by specific integral representations and their associated functional calculi;\ they were initiated in \cite{BANJOUR, CDPS1, Fivedim, Polyf1, Polyf2}, with recent $H^{\infty}$-calculus advances in \cite{MILANJPETER, MPS23, Quadratic}. The Fueter-Sce theorem gives $D\,\Delta_{n+1}^{(n-1)/2}f(x)=0$ for all slice hyperholomorphic functions on an open set $U\subseteq\mathbb{R}^{n+1}$, where $D$ is the Euclidean Dirac operator. For odd $n$, the operator $\Delta_{n+1}^{(n-1)/2}$ factorizes through operators of the form
\begin{equation}\label{T1}
   D^{\beta}\,\Delta_{n+1}^{m}, \qquad\text{and}\qquad \overline{D}^{\beta}\,\Delta_{n+1}^{m}, \qquad \beta, m\in\mathbb{N},
\end{equation}
with $\overline{D}$ the conjugate Dirac operator. The Dirac fine structures on the $S$-spectrum are, roughly, suitable classes of functions in the kernels of the operators \eqref{T1} factorizing $\Delta_{n+1}^{(n-1)/2}$, together with their functional calculi;\ the operators \eqref{T1} are themselves typical examples to which the spectral theory on the $S$-spectrum applies.

\section*{Declarations and statements}

{\bf Data availability}. The research in this paper does not imply use of data.

{\bf Conflict of interest}. The authors declare that there is no conflict of interest.


\begin{thebibliography}{99}

\bibitem{adler} S. Adler: \textit{Quaternionic Quantum Mechanics and Quaternionic Quantum Fields}. Volume 88 of \textit{International Series of Monographs on Physics}. Oxford University Press, New York (1995).

\bibitem{ACK} D. Alpay, F. Colombo, D.P. Kimsey: \textit{The spectral theorem for quaternionic unbounded normal operators based on the $S$-spectrum}. J. Math. Phys. \textbf{57}(2), 023503, 27 p. (2016).

\bibitem{ACS2016} D. Alpay, F. Colombo, I. Sabadini: \textit{Slice hyperholomorphic Schur analysis}. Volume 256 of Operator Theory: Advances and Applications, Basel, Birkhäuser/Springer, xii, 362 p. (2016).


\bibitem{AlpayColSab2020} D. Alpay, F. Colombo, I. Sabadini: \textit{Quaternionic de Branges spaces and characteristic operator function}. SpringerBriefs in Mathematics, Springer, Cham x, 116 p. (2020).

\bibitem{AlpayColSab2024} D. Alpay, F. Colombo, I. Sabadini: \textit{Quaternionic Hilbert spaces and slice hyperholomorphic functions}. Operator Theory: Advances and Applications 304. Cham: Birkhäuser, xi, 348 p. (2024).


\bibitem{MC10} P. Auscher, A. Axelsson, A. McIntosh: \textit{On a quadratic estimate related to the Kato conjecture and boundary value problems}. Harmonic analysis and partial differential equations, Contemp. Math., 505, Amer. Math. Soc., Providence, RI 105-129 (2010).

\bibitem{MC97} P. Auscher, A. McIntosh, A. Nahmod: \textit{Holomorphic functional calculi of operators, quadratic estimates and interpolation}. Indiana Univ. Math. J. \textbf{46} 375--403 (1997).

\bibitem{MC06} A. Axelsson, S. Keith, A. McIntosh: \textit{Quadratic estimates and functional calculi of perturbed Dirac operators}. Invent. Math. \textbf{163} 455--497 (2006).




\bibitem{Bar91}
C.~Bär: \textit{Das Spektrum von Dirac-Operatoren}. Dissertation, Rheinische Friedrich-Wilhelms-Universität Bonn, Bonn, 1990; Bonner Mathematische Schriften, \textbf{217}, Universität Bonn, Mathematisches Institut, Bonn, 124 pp. (1991).


\bibitem{BARACCO} L. Baracco, F. Colombo, M.M. Peloso, S. Pinton: \textit{Fractional powers of higher-order vector operators on bounded and unbounded domains}. Proc. Edinb. Math. Soc. \textbf{4}(2) 65, 912-937 (2022).

\bibitem{Behzadan2018} A. Behzadan and M. Holst: \textit{Sobolev-slobodeckij spaces on compact manifolds, revisited}, 2018.

\bibitem{beschastnyi2025s} I. Beschastnyi, F. Colombo, S. A. Lucas, I. Sabadini: \textit{The $S$-resolvent estimates for the dirac operator on hyperbolic and spherical spaces}. J. Geom. Anal. 36 (2026), no. 5, Paper No. 177.

\bibitem{BF} G. Birkhoff, J. von Neumann: \textit{The logic of quantum mechanics}. Ann. of Math. (2) \textbf{37}(4) 823--843 (1936).

\bibitem{BREZIS} H. Brezis: \textit{Functional analysis, Sobolev spaces and partial differential equations}. Universitext. Springer, New York, xiv+599 p. (2011).




\bibitem{Bbw93}
B.~Booß-Bavnbek and K.~P. Wojciechowski: \textit{Elliptic boundary problems for Dirac operators}. Mathematics: Theory \& Applications, Birkhäuser Boston, Inc., Boston, MA, 307 pp. (1993).




\bibitem{BANJOUR} F. Colombo, A. De Martino, S. Pinton: \textit{Harmonic and polyanalytic functional calculi on the $S$-spectrum for unbounded operators}. Banach J. Math. Anal. \textbf{17}, no. 4, Paper No. 84 41 p. (2023).

\bibitem{CDPS1} F. Colombo, A. De Martino, S. Pinton, I. Sabadini: \textit{Axially harmonic functions and the harmonic functional calculus on the $S$-spectrum}. J. Geom. Anal. \textbf{33}(2) 54 p. (2023).

\bibitem{Fivedim} F. Colombo, A. De Martino, S. Pinton, I. Sabadini: \textit{The fine structure of the spectral theory on the $S$-spectrum in dimension five}. J. Geom. Anal. \textbf{33}, no. 9, Paper No. 300 73 p. (2023).


\bibitem{CGdiffusion2018} F. Colombo. J. Gantner: \textit{An application of the $S$-functional calculus to fractional diffusion processes}. Milan J. Math. \textbf{86}(2), 225--303 (2018).

\bibitem{FJBOOK} F. Colombo, J. Gantner: \textit{Quaternionic closed operators, fractional powers and fractional diffusion processes}. Operator Theory: Advances and Applications \textbf{274}, Birkhäuser/Springer, Cham viii+322 (2019).

\bibitem{CGK} F. Colombo, J. Gantner, D.P. Kimsey: \textit{Spectral theory on the $S$-spectrum for quaternionic operators}. Volume 270 of Operator Theory: Advances and Applications, Birkhäuser/Springer, Cham ix+356 (2018).

\bibitem{ADVCGKS} F. Colombo, J. Gantner, D.P. Kimsey, I. Sabadini: \textit{Universality property of the $S$-functional calculus, noncommuting matrix variables and Clifford operators}. Adv. Math. \textbf{410}, Paper No. 108719, 39 p. (2022).

\bibitem{ColKim} F. Colombo, D.P. Kimsey: \textit{The spectral theorem for normal operators on a Clifford module}. Anal. Math. Phys., \textbf{12}(1), Paper No. 25, 92 (2022).

\bibitem{Gradient} F. Colombo, F. Mantovani, P. Schlosser: \textit{Spectral properties of the gradient operator with nonconstant coefficients.} Anal. Math. Phys. 14, no. 5, Paper No. 108, 31 p. (2024).

    \bibitem{Quadratic}
     F. Colombo, F. Mantovani, P. Schlosser: \textit{Quadratic estimates for the $H^\infty$-functional calculus of bisectorial Clifford operators}. J. Geom. Anal. 36 (2026), no. 2, Paper No. 46, 31 pp.

       \bibitem{fracPQ}
F. Colombo, F. Mantovani, P. Schlosser: \textit{
Nonlocal Fourier Laws for Heat Propagation via Fractional powers of Vector Operators.}
Preprint arXiv:2604.13214.


\bibitem{MILANJPETER} F. Colombo, S. Pinton, P. Schlosser: \textit{The $H^\infty$-functional calculi for the quaternionic fine structures of Dirac type}. Milan J. Math., \textbf{92}, no. 1, 73-122 (2024).

\bibitem{ColSabStrupSce} F. Colombo, I. Sabadini, D.C. Struppa: \textit{Michele Sce's works in hypercomplex analysis -- A translation with commentaries}. Birkh\"auser/Springer, Cham (2020).

\bibitem{ColomboSabadiniStruppa2011} F. Colombo, I. Sabadini, D.C. Struppa: \textit{Noncommutative functional calculus. Theory and applications of slice hyperholomorphic functions}. Volume 289 of Progress in Mathematics, Birkhäuser/Springer Basel AG, Basel vi+221 p. (2011).

\bibitem{ColSab2006} F. Colombo, I. Sabadini, D.C. Struppa: \textit{A new functional calculus for noncommuting operators}. J. Funct. Anal. \textbf{254}(8), 2255--2274 (2008).

\bibitem{Cno02}
J.~Cnops: \textit{An introduction to Dirac operators on manifolds}. Progress in Mathematical Physics, \textbf{24}, Birkhäuser Boston, Inc., Boston, MA, 211 pp. (2002).


\bibitem{Polyf1} A. De Martino, S. Pinton: \textit{A polyanalytic functional calculus of order 2 on the $S$-spectrum}. Proc. Amer. Math. Soc. \textbf{151}, 2471--2488 (2023).

\bibitem{Polyf2} A. De Martino, S. Pinton: \textit{Properties of a polyanalytic functional calculus on the $S$-spectrum}. Math. Nachr. \textbf{296} 5190-5226 (2023).

\bibitem{MPS23} A. De Martino, S. Pinton, P. Schlosser: \textit{The harmonic $H^\infty$-functional calculus based on the $S$-spectrum}. J. Spectr. Theory. \textbf{14}, no. 1, 121-162 (2024).

\bibitem{DSS} R. Delanghe, F. Sommen, V. Sou\v cek: \textit{Clifford algebra and spinor-valued functions.} Mathematics and its Applications \textbf{53}, Kluwer Academic Publishers Group, Dordrecht (1992).


\bibitem{Esp98}
G.~Esposito: \textit{Dirac operators and spectral geometry}. Cambridge Lecture Notes in Physics, \textbf{12}, Cambridge University Press, Cambridge, 209 pp. (1998).



\bibitem{Fri00}
T.~Friedrich: \textit{Dirac operators in Riemannian geometry}. Translated from the 1997 German original by Andreas Nestke, Graduate Studies in Mathematics, \textbf{25}, American Mathematical Society, Providence, RI, 195 pp. (2000).



\bibitem{MC98} E. Franks, A. McIntosh: \textit{Discrete quadratic estimates and holomorphic functional calculi in Banach spaces}. Bull. Austral. Math. Soc. \textbf{58}, 271--290 (1998).

\bibitem{Fueter} R. Fueter: \textit{Die Funktionentheorie der Differentialgleichungen $\Delta u=0$ und $\Delta\Delta u=0$ mit vier reellen Variablen}. Comment. Math. Helv. \textbf{7}(1), 307--330 (1934).

\bibitem{JONAMEM} J. Gantner: \textit{Operator theory on one-sided quaternion linear spaces: intrinsic S-functional calculus and spectral operators.} Mem. Amer. Math. Soc. \textbf{267}, no. 1297, iii+101 p. (2020).

\bibitem{JONADIRECT} J. Gantner: \textit{A direct approach to the $S$-functional calculus for closed operators}. J. Operator Theory. \textbf{77} no.2, 287-331 (2017).

\bibitem{JONAQSTUD} J. Gantner: \textit{On the equivalence of complex and quaternionic quantum mechanics}. Quantum Stud. Math. Found. 5, No. 2, 357-390 (2018).

\bibitem{gilbert1991clifford} J. E. Gilbert and M. Murray: \textit{Clifford algebras and Dirac operators in harmonic analysis}. Number 26. Cambridge University Press, 1991.


\bibitem{Gin09}
N.~Ginoux: \textit{The Dirac spectrum}. Lecture Notes in Mathematics, \textbf{1976}, Springer-Verlag, Berlin, 156 pp. (2009).


\bibitem{Haase} M. Haase: \textit{The functional calculus for sectorial operators}. Operator Theory: Advances and Applications \textbf{169}, Birkhäuser Verlag, Basel (2006).

\bibitem{HYTONBOOK1} T. Hytönen, J. van Neerven, M. Veraar, L. Weis: \textit{Analysis in Banach spaces. Vol. II. Probabilistic methods and operator theory}. Ergebnisse der Mathematik und ihrer Grenzgebiete. 3. Folge. A Series of Modern Surveys in Mathematics \textbf{67}, Springer, Cham xxi+616 p. (2017).

\bibitem{HYTONBOOK2} T. Hytönen, J. van Neerven, M. Veraar, and L. Weis: \textit{Analysis in Banach spaces. Vol. I. Martingales and Littlewood-Paley theory}. Ergebnisse der Mathematik und ihrer Grenzgebiete. 3. Folge. A Series of Modern Surveys in Mathematics \textbf{63}, Springer, Cham xvi+614 p. (2016).


\bibitem{Hua06}
J.-S. Huang and P.~Pandžić: \textit{Dirac operators in representation theory}. Mathematics: Theory \& Applications, Birkhäuser Boston, Inc., Boston, MA, 199 pp. (2006).

    \bibitem{mantovani}
     F. Mantovani: \textit{Slice hyperholomorphicity of the S-resolvent operators and boundary conditions}. Complex Anal. Oper. Theory 20 (2026), no. 3, Paper No. 66, 24 pp.

     \bibitem{bisectorial}
     F. Mantovani, P. Schlosser: \textit{The $H^\infty$-functional calculus for bisectorial Clifford operators}.
     J. Spectr. Theory 15 (2025), no. 2, 751–818.


\bibitem{JBOOK} B. Jefferies: \textit{Spectral properties of noncommuting operators}. Volume 1843, Lecture Notes in Mathematics. Springer-Verlag, Berlin (2004).

\bibitem{JM} B. Jefferies, A. McIntosh, J. Picton-Warlow:\textit{The monogenic functional calculus}. Studia Math. \textbf{136}, 99-119 (1999).

\bibitem{lawsonmichelson_spingeometry} H. B. Lawson and M. L. Michelsohn: \textit{Spin Geometry (PMS-38)}. Princeton University Press, 1989.

\bibitem{lee2012introduction} J. Lee: \textit{Introduction to Smooth Manifolds}. Graduate Texts in Mathematics. Springer New York, 2012.

\bibitem{TAOBOOK} P. Li, T. Qian: \textit{Singular integrals and Fourier theory on Lipschitz boundaries}. Science Press Beijing, Beijing, Springer, Singapore xv+306 p. (2019).


\bibitem{M87} J. Marschall: \textit{The trace of Sobolev-Slobodeckij spaces on Lipschitz domains}. Manuscripta Math. \textbf{58}, 47-65 (1987).

\bibitem{McI1} A. McIntosh: \textit{Operators which have an $H^\infty$ functional calculus}. Miniconference on operator theory and partial differential equations (North Ryde, 1986), 210--231, Proc. Centre Math. Anal. Austral. Nat. Univ., 14, Austral. Nat. Univ., Canberra (1986).


\bibitem{TaoQian1} T. Qian: \textit{Generalization of Fueter's result to $\textbf{R}^{n+1}$}. Atti Accad. Naz. Lincei Cl. Sci. Fis. Mat. Natur. Rend. Lincei (9) Mat. Appl. \textbf{8}(2), 111--117 (1997).


\bibitem{Rod16}
W.~A. Rodrigues, Jr. and E.~Capelas de Oliveira: \textit{The many faces of Maxwell, Dirac and Einstein equations. A Clifford bundle approach}. Second edition, Lecture Notes in Physics, \textbf{922}, Springer, [Cham], 587 pp. (2016).

\bibitem{Sce} M. Sce: \textit{Osservazioni sulle serie di potenze nei moduli quadratici}. Atti Accad. Naz. Lincei Rend. Cl. Sci. Fis. Mat. Nat. \textbf{23}(8), 220--225 (1957).





\end{thebibliography}
\end{document}